\documentclass[a4paper]{amsart}
\usepackage{graphicx,color}
\usepackage{pict2e}
\usepackage{amsmath}
\usepackage{amssymb}
\usepackage{amsbsy}
\usepackage{amsgen}
\usepackage{amsopn}
\usepackage{amscd}
\usepackage{amstext}
\usepackage{amsthm}

\theoremstyle{definition}

\theoremstyle{plain}
\newtheorem{thm}{Theorem}[section]

\newtheorem{lem}{Lemma}[section]

\theoremstyle{remark}
\newtheorem{rem}{Remark}[section]

\theoremstyle{definition}

\begin{document}
\title{Chain homotopy maps and a universal differential for Khovanov-type homology}
\author{Noboru Ito}\thanks{The author is a Research Fellow of the Japan Society for the Promotion of Science.  This work was partly supported by KAKENHI}
\maketitle

\begin{abstract}
We give chain homotopy maps of Khovanov-type link homology of a universal differential.  The universal differential, discussed by Mikhail Khovanov, Marco Mackaay, Paul Turner and Pedro Vaz, contains the original Khovanov's differential and Lee's differential.  
We also consider the conditions of any differential ensuring the Reidemeister invariance for the chain homotopy maps.  
\end{abstract}

\section{Introduction.  }
\subsection{Motivation.  }
The aim of this note is to show that there exists certain homotopy maps ensuring the Reidemeister invariance of Khovanov-type homology of a universal differential.  
By using Viro's construction \cite{viro, jacobsson} of the Khovanov homology $\mathcal{H}^{i,j}$ \cite{khovanov}, a universal differential \cite{khovanovF, turner, MTV} $\mathcal{C}^{i}$ $\to$ $\mathcal{C}^{i+1}$ on Khovanov complex $\mathcal{C}^{i}$ $:=$ $\oplus_{j}C^{i,j}$ \cite{jacobsson} over the ring $\mathbb{Z}[s, t]$ can be defined as 
\begin{equation}
\delta_{s, t}(S \otimes [x]) := \sum_{T}T \otimes [xa] 
\end{equation}
such that
\begin{equation}\label{dif}
\begin{minipage}{50pt}
\begin{picture}(50,50)
\put(0,10){\line(1,1){30}}
\qbezier(30,10)(30,10)(20,20)
\qbezier(10,30)(0,40)(0,40)
\linethickness{2pt}
\put(15,15){\circle*{5}}
\put(15,35){\circle*{5}}
\put(15,15){\line(0,1){20}}
\put(22,23){$q$}
\put(3,23){$p$}
\put(13,5){$a$}
\end{picture}
\end{minipage} \otimes [x] \qquad
\longmapsto \qquad\quad\quad
\begin{minipage}{50pt}
\begin{picture}(50,50)
\put(0,10){\line(1,1){30}}
\qbezier(30,10)(30,10)(20,20)
\qbezier(10,30)(0,40)(0,40)
\linethickness{2pt}
\put(5,25){\circle*{5}}
\put(25,25){\circle*{5}}
\put(5,25){\line(1,0){20}}
\put(5,40){$q:p$}
\put(5,5){$p:q$}
\put(30,23){$a$}
\end{picture}
\end{minipage}\ \  \otimes [xa]
\end{equation}
where $a$ is a crossing of a link diagram, $p$ and $q$ denote signs ($+$ or $-$) and $p:q$ and $q:p$ are defined by Figure \ref{frobenius}.  

This differential follows a universal Frobenius algebra \cite{khovanovF}.  As in \cite{turner, MTV}, the differential possess the following property.  
When $s$ $=$ $t$ $=$ $0$, we get the original Khovanov theory \cite{khovanov} and when $(s, t)$ $=$ $(0,1)$, Lee theory \cite{lee} introducing Rasmussen's invariants \cite{rasmussen}.  
When we choose the coefficient $\mathbb{Z}/2[s]$ (take $t$ $=$ $0$), the theory is known as Bar-Natan theory \cite{bar-natan}.  
On the other hand, there exist the explicit chain homotopy maps \cite{viro, ito} for the original Khovanov homology.  The chain homotopy maps ensuring Reidemeister invariance of Khovanov homology are given by \cite{viro} for the first Reidemeister move, by \cite{ito} for the other Reidemeister moves.  

Here, the following question naturally arise.  Replacing the differential of the original Khovanov homology \cite{khovanov} by the universal differential, what are chain homotopy maps and retractions needed for ensuring the Reidemeister invariance? 
When the homotopy maps \cite{viro, ito} are given, is there any explicit condition for a universal differential on the Khovanov complex to induce the Reidemeister invariance of a family of homology?  

In this note, we discuss these problem in detail.  We arrive at the result (Theorem \ref{t1} and \ref{t2}) as mentioned at the end of this note.

\subsection{Definition and Notation.  }

We use the definition of Khovanov homology defined in \cite[Section 2]{jacobsson} except for replacing the Frobenius calculus \cite[Figure 3]{jacobsson} in the definition of the differential $d$ by Figure \ref{frobenius}.  Thus let us use the same symbols $p:q$ and $q:p$ as \cite[Figure 3]{jacobsson} in a generalised meaning as follows: for a given signs $p$ and $q$, the signs $p:q$ and $q:p$ of circles are defined by Figure \ref{frobenius}.  We denote the generalised differential by $\delta_{s,t}$.  $\operatorname{m}(p, q)$ denotes $p:q$ in (\ref{f1})--(\ref{f2}) for signs $p$, $q$.  

\begin{rem}
Note that the symbols $p:q$ and $q:p$ are generalised.  
In fact, considering the case $s$ $=$ $t$ $=$ $0$ (resp. $s$ $=$ $0$ and $t$ $=$ $1$), the differential give the original Khovanov homology (resp. Lee's differential).  
\end{rem}

\begin{figure}
\begin{align}
\begin{minipage}{90pt}\label{f1}
         \begin{picture}(90,50)
         \put(20,10){\line(1,0){1}}
        \put(17.5,10.2){\line(-4,1){1}}
        \put(15,10.9){\line(-3,1){1}}
        \put(12.5,12){\line(-2,1){1}}
        \put(10.6,13.2){\line(-1,1){.9}}
        \put(8.9,14.9){\line(-3,4){1}}
        \put(7.3,17){\line(-1,2){.8}}
        \put(6,20){\line(-1,3){.5}}
        \put(5.3,22.7){\line(-1,6){.2}}
        \put(5,25){\line(0,1){1}}
        \put(20,40){\line(-1,0){1}}
        \put(17.5,39.8){\line(-4,-1){1}}
        \put(15,39.1){\line(-3,-1){1}}
        \put(12.5,38){\line(-2,-1){1}}
        \put(10.6,36.8){\line(-1,-1){.9}}
        \put(8.9,35.1){\line(-3,-4){1}}
        \put(7.3,33){\line(-1,-2){.8}}
        \put(6,30){\line(-1,-3){.5}}
        \put(5.3,27.3){\line(-1,-6){.2}}
        \put(20, 23){$+$}
                \put(20,25){\oval(30,30)[r]}
                \put(70,25){\oval(30,30)[l]}
        \put(62,23){$+$}
                 \put(70,10){\line(1,0){1}}
        \put(72.5,10.2){\line(4,1){1}}
        \put(75,10.9){\line(3,1){1}}
        \put(77.5,12){\line(2,1){1}}
        \put(79.4,13.2){\line(1,1){.9}}
        \put(81.1,14.9){\line(3,4){1}}
        \put(82.7,17){\line(1,2){.8}}
        \put(84,20){\line(1,3){.5}}
        \put(84.7,22.7){\line(1,6){.2}}
        \put(85,25){\line(0,1){1}}
        \put(70,40){\line(1,0){1}}
        \put(72.5,39.8){\line(4,-1){1}}
        \put(75,39.1){\line(3,-1){1}}
        \put(77.5,38){\line(2,-1){1}}
        \put(79.4,36.8){\line(1,-1){.9}}
        \put(81.1,35.1){\line(3,-4){1}}
        \put(82.7,33){\line(1,-2){.8}}
        \put(84,30){\line(1,-3){.5}}
        \put(84.7,27.3){\line(1,-6){.2}}
        \end{picture}
        \end{minipage}
&\begin{minipage}{90pt}
\begin{picture}(90,50)
\put(10,24){\vector(1,0){50}}
\end{picture}
        \end{minipage}
\begin{minipage}{70pt}
\begin{picture}(70,50)
\put(-11,22){$s$}
\put(25,22){$+$}
\put(15,10){\line(1,0){1}}
        \put(12.5,10.2){\line(-4,1){1}}
        \put(10,10.9){\line(-3,1){1}}
        \put(7.5,12){\line(-2,1){1}}
        \put(5.6,13.2){\line(-1,1){.9}}
        \put(3.9,14.9){\line(-3,4){1}}
        \put(2.3,17){\line(-1,2){.8}}
        \put(1,20){\line(-1,3){.5}}
        \put(0.3,22.7){\line(-1,6){.2}}
        \put(0,25){\line(0,1){1}}
        \put(15,40){\line(-1,0){1}}
        \put(12.5,39.8){\line(-4,-1){1}}
        \put(10,39.1){\line(-3,-1){1}}
        \put(7.5,38){\line(-2,-1){1}}
        \put(5.6,36.8){\line(-1,-1){.9}}
        \put(3.9,35.1){\line(-3,-4){1}}
        \put(2.3,33){\line(-1,-2){.8}}
        \put(1,30){\line(-1,-3){.5}}
        \put(0.3,27.3){\line(-1,-6){.2}}
\qbezier(15,40)(30,25)(45,40)
\qbezier(15,10)(30,25)(45,10)
 \put(45,10){\line(1,0){1}}
        \put(47.5,10.2){\line(4,1){1}}
        \put(50,10.9){\line(3,1){1}}
        \put(52.5,12){\line(2,1){1}}
        \put(54.4,13.2){\line(1,1){.9}}
        \put(56.1,14.9){\line(3,4){1}}
        \put(57.7,17){\line(1,2){.8}}
        \put(59,20){\line(1,3){.5}}
        \put(59.7,22.7){\line(1,6){.2}}
        \put(60,25){\line(0,1){1}}
        \put(45,40){\line(1,0){1}}
        \put(47.5,39.8){\line(4,-1){1}}
        \put(50,39.1){\line(3,-1){1}}
        \put(52.5,38){\line(2,-1){1}}
        \put(54.4,36.8){\line(1,-1){.9}}
        \put(56.1,35.1){\line(3,-4){1}}
        \put(57.7,33){\line(1,-2){.8}}
        \put(59,30){\line(1,-3){.5}}
        \put(59.7,27.3){\line(1,-6){.2}}
        \end{picture}
    \end{minipage}
    \text{$+$}\quad
    \begin{minipage}{80pt}
\begin{picture}(80,50)
\put(-8,22){$t$}
\put(15,10){\line(1,0){1}}
\put(25,22){$-$}
        \put(12.5,10.2){\line(-4,1){1}}
        \put(10,10.9){\line(-3,1){1}}
        \put(7.5,12){\line(-2,1){1}}
        \put(5.6,13.2){\line(-1,1){.9}}
        \put(3.9,14.9){\line(-3,4){1}}
        \put(2.3,17){\line(-1,2){.8}}
        \put(1,20){\line(-1,3){.5}}
        \put(0.3,22.7){\line(-1,6){.2}}
        \put(0,25){\line(0,1){1}}
        \put(15,40){\line(-1,0){1}}
        \put(12.5,39.8){\line(-4,-1){1}}
        \put(10,39.1){\line(-3,-1){1}}
        \put(7.5,38){\line(-2,-1){1}}
        \put(5.6,36.8){\line(-1,-1){.9}}
        \put(3.9,35.1){\line(-3,-4){1}}
        \put(2.3,33){\line(-1,-2){.8}}
        \put(1,30){\line(-1,-3){.5}}
        \put(0.3,27.3){\line(-1,-6){.2}}
\qbezier(15,40)(30,25)(45,40)
\qbezier(15,10)(30,25)(45,10)
 \put(45,10){\line(1,0){1}}
        \put(47.5,10.2){\line(4,1){1}}
        \put(50,10.9){\line(3,1){1}}
        \put(52.5,12){\line(2,1){1}}
        \put(54.4,13.2){\line(1,1){.9}}
        \put(56.1,14.9){\line(3,4){1}}
        \put(57.7,17){\line(1,2){.8}}
        \put(59,20){\line(1,3){.5}}
        \put(59.7,22.7){\line(1,6){.2}}
        \put(60,25){\line(0,1){1}}
        \put(45,40){\line(1,0){1}}
        \put(47.5,39.8){\line(4,-1){1}}
        \put(50,39.1){\line(3,-1){1}}
        \put(52.5,38){\line(2,-1){1}}
        \put(54.4,36.8){\line(1,-1){.9}}
        \put(56.1,35.1){\line(3,-4){1}}
        \put(57.7,33){\line(1,-2){.8}}
        \put(59,30){\line(1,-3){.5}}
        \put(59.7,27.3){\line(1,-6){.2}}
        \end{picture}
    \end{minipage}\\
    \begin{minipage}{90pt}\label{f2}
         \begin{picture}(90,50)
         \put(20,10){\line(1,0){1}}
        \put(17.5,10.2){\line(-4,1){1}}
        \put(15,10.9){\line(-3,1){1}}
        \put(12.5,12){\line(-2,1){1}}
        \put(10.6,13.2){\line(-1,1){.9}}
        \put(8.9,14.9){\line(-3,4){1}}
        \put(7.3,17){\line(-1,2){.8}}
        \put(6,20){\line(-1,3){.5}}
        \put(5.3,22.7){\line(-1,6){.2}}
        \put(5,25){\line(0,1){1}}
        \put(20,40){\line(-1,0){1}}
        \put(17.5,39.8){\line(-4,-1){1}}
        \put(15,39.1){\line(-3,-1){1}}
        \put(12.5,38){\line(-2,-1){1}}
        \put(10.6,36.8){\line(-1,-1){.9}}
        \put(8.9,35.1){\line(-3,-4){1}}
        \put(7.3,33){\line(-1,-2){.8}}
        \put(6,30){\line(-1,-3){.5}}
        \put(5.3,27.3){\line(-1,-6){.2}}
        \put(20, 23){$+$}
                \put(20,25){\oval(30,30)[r]}
                \put(70,25){\oval(30,30)[l]}
        \put(62,23){$-$}
                 \put(70,10){\line(1,0){1}}
        \put(72.5,10.2){\line(4,1){1}}
        \put(75,10.9){\line(3,1){1}}
        \put(77.5,12){\line(2,1){1}}
        \put(79.4,13.2){\line(1,1){.9}}
        \put(81.1,14.9){\line(3,4){1}}
        \put(82.7,17){\line(1,2){.8}}
        \put(84,20){\line(1,3){.5}}
        \put(84.7,22.7){\line(1,6){.2}}
        \put(85,25){\line(0,1){1}}
        \put(70,40){\line(1,0){1}}
        \put(72.5,39.8){\line(4,-1){1}}
        \put(75,39.1){\line(3,-1){1}}
        \put(77.5,38){\line(2,-1){1}}
        \put(79.4,36.8){\line(1,-1){.9}}
        \put(81.1,35.1){\line(3,-4){1}}
        \put(82.7,33){\line(1,-2){.8}}
        \put(84,30){\line(1,-3){.5}}
        \put(84.7,27.3){\line(1,-6){.2}}
        \end{picture}
        \end{minipage}
&\begin{minipage}{90pt}
\begin{picture}(90,50)
\put(10,24){\vector(1,0){50}}
\end{picture}
        \end{minipage}
    \begin{minipage}{70pt}
\begin{picture}(70,50)
\put(25,22){$+$}
\put(15,10){\line(1,0){1}}
        \put(12.5,10.2){\line(-4,1){1}}
        \put(10,10.9){\line(-3,1){1}}
        \put(7.5,12){\line(-2,1){1}}
        \put(5.6,13.2){\line(-1,1){.9}}
        \put(3.9,14.9){\line(-3,4){1}}
        \put(2.3,17){\line(-1,2){.8}}
        \put(1,20){\line(-1,3){.5}}
        \put(0.3,22.7){\line(-1,6){.2}}
        \put(0,25){\line(0,1){1}}
        \put(15,40){\line(-1,0){1}}
        \put(12.5,39.8){\line(-4,-1){1}}
        \put(10,39.1){\line(-3,-1){1}}
        \put(7.5,38){\line(-2,-1){1}}
        \put(5.6,36.8){\line(-1,-1){.9}}
        \put(3.9,35.1){\line(-3,-4){1}}
        \put(2.3,33){\line(-1,-2){.8}}
        \put(1,30){\line(-1,-3){.5}}
        \put(0.3,27.3){\line(-1,-6){.2}}
\qbezier(15,40)(30,25)(45,40)
\qbezier(15,10)(30,25)(45,10)
 \put(45,10){\line(1,0){1}}
        \put(47.5,10.2){\line(4,1){1}}
        \put(50,10.9){\line(3,1){1}}
        \put(52.5,12){\line(2,1){1}}
        \put(54.4,13.2){\line(1,1){.9}}
        \put(56.1,14.9){\line(3,4){1}}
        \put(57.7,17){\line(1,2){.8}}
        \put(59,20){\line(1,3){.5}}
        \put(59.7,22.7){\line(1,6){.2}}
        \put(60,25){\line(0,1){1}}
        \put(45,40){\line(1,0){1}}
        \put(47.5,39.8){\line(4,-1){1}}
        \put(50,39.1){\line(3,-1){1}}
        \put(52.5,38){\line(2,-1){1}}
        \put(54.4,36.8){\line(1,-1){.9}}
        \put(56.1,35.1){\line(3,-4){1}}
        \put(57.7,33){\line(1,-2){.8}}
        \put(59,30){\line(1,-3){.5}}
        \put(59.7,27.3){\line(1,-6){.2}}
        \end{picture}
    \end{minipage}\\
    \begin{minipage}{90pt}\label{f3}
         \begin{picture}(90,50)
         \put(20,10){\line(1,0){1}}
        \put(17.5,10.2){\line(-4,1){1}}
        \put(15,10.9){\line(-3,1){1}}
        \put(12.5,12){\line(-2,1){1}}
        \put(10.6,13.2){\line(-1,1){.9}}
        \put(8.9,14.9){\line(-3,4){1}}
        \put(7.3,17){\line(-1,2){.8}}
        \put(6,20){\line(-1,3){.5}}
        \put(5.3,22.7){\line(-1,6){.2}}
        \put(5,25){\line(0,1){1}}
        \put(20,40){\line(-1,0){1}}
        \put(17.5,39.8){\line(-4,-1){1}}
        \put(15,39.1){\line(-3,-1){1}}
        \put(12.5,38){\line(-2,-1){1}}
        \put(10.6,36.8){\line(-1,-1){.9}}
        \put(8.9,35.1){\line(-3,-4){1}}
        \put(7.3,33){\line(-1,-2){.8}}
        \put(6,30){\line(-1,-3){.5}}
        \put(5.3,27.3){\line(-1,-6){.2}}
        \put(20, 23){$-$}
                \put(20,25){\oval(30,30)[r]}
                \put(70,25){\oval(30,30)[l]}
        \put(62,23){$+$}
                 \put(70,10){\line(1,0){1}}
        \put(72.5,10.2){\line(4,1){1}}
        \put(75,10.9){\line(3,1){1}}
        \put(77.5,12){\line(2,1){1}}
        \put(79.4,13.2){\line(1,1){.9}}
        \put(81.1,14.9){\line(3,4){1}}
        \put(82.7,17){\line(1,2){.8}}
        \put(84,20){\line(1,3){.5}}
        \put(84.7,22.7){\line(1,6){.2}}
        \put(85,25){\line(0,1){1}}
        \put(70,40){\line(1,0){1}}
        \put(72.5,39.8){\line(4,-1){1}}
        \put(75,39.1){\line(3,-1){1}}
        \put(77.5,38){\line(2,-1){1}}
        \put(79.4,36.8){\line(1,-1){.9}}
        \put(81.1,35.1){\line(3,-4){1}}
        \put(82.7,33){\line(1,-2){.8}}
        \put(84,30){\line(1,-3){.5}}
        \put(84.7,27.3){\line(1,-6){.2}}
        \end{picture}
        \end{minipage}
&\begin{minipage}{90pt}
\begin{picture}(90,50)
\put(10,24){\vector(1,0){50}}
\end{picture}
        \end{minipage}
    \begin{minipage}{70pt}
\begin{picture}(70,50)
\put(25,22){$+$}
\put(15,10){\line(1,0){1}}
        \put(12.5,10.2){\line(-4,1){1}}
        \put(10,10.9){\line(-3,1){1}}
        \put(7.5,12){\line(-2,1){1}}
        \put(5.6,13.2){\line(-1,1){.9}}
        \put(3.9,14.9){\line(-3,4){1}}
        \put(2.3,17){\line(-1,2){.8}}
        \put(1,20){\line(-1,3){.5}}
        \put(0.3,22.7){\line(-1,6){.2}}
        \put(0,25){\line(0,1){1}}
        \put(15,40){\line(-1,0){1}}
        \put(12.5,39.8){\line(-4,-1){1}}
        \put(10,39.1){\line(-3,-1){1}}
        \put(7.5,38){\line(-2,-1){1}}
        \put(5.6,36.8){\line(-1,-1){.9}}
        \put(3.9,35.1){\line(-3,-4){1}}
        \put(2.3,33){\line(-1,-2){.8}}
        \put(1,30){\line(-1,-3){.5}}
        \put(0.3,27.3){\line(-1,-6){.2}}
\qbezier(15,40)(30,25)(45,40)
\qbezier(15,10)(30,25)(45,10)
 \put(45,10){\line(1,0){1}}
        \put(47.5,10.2){\line(4,1){1}}
        \put(50,10.9){\line(3,1){1}}
        \put(52.5,12){\line(2,1){1}}
        \put(54.4,13.2){\line(1,1){.9}}
        \put(56.1,14.9){\line(3,4){1}}
        \put(57.7,17){\line(1,2){.8}}
        \put(59,20){\line(1,3){.5}}
        \put(59.7,22.7){\line(1,6){.2}}
        \put(60,25){\line(0,1){1}}
        \put(45,40){\line(1,0){1}}
        \put(47.5,39.8){\line(4,-1){1}}
        \put(50,39.1){\line(3,-1){1}}
        \put(52.5,38){\line(2,-1){1}}
        \put(54.4,36.8){\line(1,-1){.9}}
        \put(56.1,35.1){\line(3,-4){1}}
        \put(57.7,33){\line(1,-2){.8}}
        \put(59,30){\line(1,-3){.5}}
        \put(59.7,27.3){\line(1,-6){.2}}
        \end{picture}
    \end{minipage}\\
    \begin{minipage}{90pt}\label{f4}
         \begin{picture}(90,50)
         \put(20,10){\line(1,0){1}}
        \put(17.5,10.2){\line(-4,1){1}}
        \put(15,10.9){\line(-3,1){1}}
        \put(12.5,12){\line(-2,1){1}}
        \put(10.6,13.2){\line(-1,1){.9}}
        \put(8.9,14.9){\line(-3,4){1}}
        \put(7.3,17){\line(-1,2){.8}}
        \put(6,20){\line(-1,3){.5}}
        \put(5.3,22.7){\line(-1,6){.2}}
        \put(5,25){\line(0,1){1}}
        \put(20,40){\line(-1,0){1}}
        \put(17.5,39.8){\line(-4,-1){1}}
        \put(15,39.1){\line(-3,-1){1}}
        \put(12.5,38){\line(-2,-1){1}}
        \put(10.6,36.8){\line(-1,-1){.9}}
        \put(8.9,35.1){\line(-3,-4){1}}
        \put(7.3,33){\line(-1,-2){.8}}
        \put(6,30){\line(-1,-3){.5}}
        \put(5.3,27.3){\line(-1,-6){.2}}
        \put(20, 23){$-$}
                \put(20,25){\oval(30,30)[r]}
                \put(70,25){\oval(30,30)[l]}
        \put(62,23){$-$}
                 \put(70,10){\line(1,0){1}}
        \put(72.5,10.2){\line(4,1){1}}
        \put(75,10.9){\line(3,1){1}}
        \put(77.5,12){\line(2,1){1}}
        \put(79.4,13.2){\line(1,1){.9}}
        \put(81.1,14.9){\line(3,4){1}}
        \put(82.7,17){\line(1,2){.8}}
        \put(84,20){\line(1,3){.5}}
        \put(84.7,22.7){\line(1,6){.2}}
        \put(85,25){\line(0,1){1}}
        \put(70,40){\line(1,0){1}}
        \put(72.5,39.8){\line(4,-1){1}}
        \put(75,39.1){\line(3,-1){1}}
        \put(77.5,38){\line(2,-1){1}}
        \put(79.4,36.8){\line(1,-1){.9}}
        \put(81.1,35.1){\line(3,-4){1}}
        \put(82.7,33){\line(1,-2){.8}}
        \put(84,30){\line(1,-3){.5}}
        \put(84.7,27.3){\line(1,-6){.2}}
        \end{picture}
        \end{minipage}
&\begin{minipage}{90pt}
\begin{picture}(90,50)
\put(10,24){\vector(1,0){50}}
\end{picture}
        \end{minipage}
    \begin{minipage}{70pt}
\begin{picture}(70,50)
\put(25,22){$-$}
\put(15,10){\line(1,0){1}}
        \put(12.5,10.2){\line(-4,1){1}}
        \put(10,10.9){\line(-3,1){1}}
        \put(7.5,12){\line(-2,1){1}}
        \put(5.6,13.2){\line(-1,1){.9}}
        \put(3.9,14.9){\line(-3,4){1}}
        \put(2.3,17){\line(-1,2){.8}}
        \put(1,20){\line(-1,3){.5}}
        \put(0.3,22.7){\line(-1,6){.2}}
        \put(0,25){\line(0,1){1}}
        \put(15,40){\line(-1,0){1}}
        \put(12.5,39.8){\line(-4,-1){1}}
        \put(10,39.1){\line(-3,-1){1}}
        \put(7.5,38){\line(-2,-1){1}}
        \put(5.6,36.8){\line(-1,-1){.9}}
        \put(3.9,35.1){\line(-3,-4){1}}
        \put(2.3,33){\line(-1,-2){.8}}
        \put(1,30){\line(-1,-3){.5}}
        \put(0.3,27.3){\line(-1,-6){.2}}
\qbezier(15,40)(30,25)(45,40)
\qbezier(15,10)(30,25)(45,10)
 \put(45,10){\line(1,0){1}}
        \put(47.5,10.2){\line(4,1){1}}
        \put(50,10.9){\line(3,1){1}}
        \put(52.5,12){\line(2,1){1}}
        \put(54.4,13.2){\line(1,1){.9}}
        \put(56.1,14.9){\line(3,4){1}}
        \put(57.7,17){\line(1,2){.8}}
        \put(59,20){\line(1,3){.5}}
        \put(59.7,22.7){\line(1,6){.2}}
        \put(60,25){\line(0,1){1}}
        \put(45,40){\line(1,0){1}}
        \put(47.5,39.8){\line(4,-1){1}}
        \put(50,39.1){\line(3,-1){1}}
        \put(52.5,38){\line(2,-1){1}}
        \put(54.4,36.8){\line(1,-1){.9}}
        \put(56.1,35.1){\line(3,-4){1}}
        \put(57.7,33){\line(1,-2){.8}}
        \put(59,30){\line(1,-3){.5}}
        \put(59.7,27.3){\line(1,-6){.2}}
        \end{picture}
    \end{minipage}\\ 
\begin{minipage}{70pt}\label{f5}
        \begin{picture}(50,60)
        \put(20,25){$+$}
        \put(25,0){\line(-1,0){1}}
        \put(22.6,0.2){\line(-4,1){1}}
        \put(20,0.9){\line(-3,1){1}}
        \put(17.5,2){\line(-2,1){1}}
        \put(15.6,3.2){\line(-1,1){.9}}
        \put(13.9,4.9){\line(-3,4){1}}
        \put(12.3,7){\line(-1,2){.8}}
        \put(11,10){\line(-1,3){.5}}
        \put(10.3,12.7){\line(-1,6){.2}}
        \put(10,15){\line(0,1){1}}
        \put(25,0){\line(1,0){1}}
        \put(27.5,0.2){\line(4,1){1}}
        \put(30,0.9){\line(3,1){1}}
        \put(32.5,2){\line(2,1){1}}
        \put(34.4,3.2){\line(1,1){.9}}
        \put(36.1,4.9){\line(3,4){1}}
        \put(37.7,7){\line(1,2){.8}}
        \put(39,10){\line(1,3){.5}}
        \put(39.7,12.7){\line(1,6){.2}}
        \put(40,15){\line(0,1){1}}
    \qbezier(10,40)(25,25)(10,15)
\qbezier(40,15)(25,25)(40,40)
        \put(25,55){\line(-1,0){1}}
        \put(22.6,54.8){\line(-4,-1){1}}
        \put(20,54.1){\line(-3,-1){1}}
        \put(17.5,53){\line(-2,-1){1}}
        \put(15.6,51.8){\line(-1,-1){.9}}
        \put(13.9,50.1){\line(-3,-4){1}}
        \put(12.3,48){\line(-1,-2){.8}}
        \put(11,45){\line(-1,-3){.5}}
        \put(10.3,42.3){\line(-1,-6){.2}}
        \put(10,40){\line(0,-1){1}}
        \put(25,55){\line(1,0){1}}
        \put(27.5,54.8){\line(4,-1){1}}
        \put(30,54.1){\line(3,-1){1}}
        \put(32.5,53){\line(2,-1){1}}
        \put(34.4,51.8){\line(1,-1){.9}}
        \put(36.1,50.1){\line(3,-4){1}}
        \put(37.7,48){\line(1,-2){.8}}
        \put(39,45){\line(1,-3){.5}}
        \put(39.7,42.3){\line(1,-6){.2}}
        \put(40,40){\line(0,-1){1}}
        \end{picture}
    \end{minipage} 
&\begin{minipage}{80pt}
        \begin{picture}(70,60)
\put(0,24){\vector(1,0){50}}
\end{picture}
    \end{minipage}
\begin{minipage}{50pt}
        \begin{picture}(50,60)
        \put(15,10){$+$}
        \put(15,40){$+$}
\put(20,0){\line(-1,0){1}}
        \put(17.6,0.2){\line(-4,1){1}}
        \put(15,0.9){\line(-3,1){1}}
        \put(12.5,2){\line(-2,1){1}}
        \put(10.6,3.2){\line(-1,1){.9}}
        \put(8.9,4.9){\line(-3,4){1}}
        \put(7.3,7){\line(-1,2){.8}}
        \put(6,10){\line(-1,3){.5}}
        \put(5.3,12.7){\line(-1,6){.2}}
        \put(5,15){\line(0,1){1}}
        \put(20,0){\line(1,0){1}}
        \put(22.5,0.2){\line(4,1){1}}
        \put(25,0.9){\line(3,1){1}}
        \put(27.5,2){\line(2,1){1}}
        \put(29.4,3.2){\line(1,1){.9}}
        \put(31.1,4.9){\line(3,4){1}}
        \put(32.7,7){\line(1,2){.8}}
        \put(34,10){\line(1,3){.5}}
        \put(34.7,12.7){\line(1,6){.2}}
        \put(35,15){\line(0,1){1}}
\qbezier(5,40)(20,20)(35,40)
\qbezier(5,15)(20,35)(35,15)
\put(20,55){\line(-1,0){1}}
        \put(17.6,54.8){\line(-4,-1){1}}
        \put(15,54.1){\line(-3,-1){1}}
        \put(12.5,53){\line(-2,-1){1}}
        \put(10.6,51.8){\line(-1,-1){.9}}
        \put(8.9,50.1){\line(-3,-4){1}}
        \put(7.3,48){\line(-1,-2){.8}}
        \put(6,45){\line(-1,-3){.5}}
        \put(5.3,42.3){\line(-1,-6){.2}}
        \put(5,40){\line(0,-1){1}}
        \put(20,55){\line(1,0){1}}
        \put(22.5,54.8){\line(4,-1){1}}
        \put(25,54.1){\line(3,-1){1}}
        \put(27.5,53){\line(2,-1){1}}
        \put(29.4,51.8){\line(1,-1){.9}}
        \put(31.1,50.1){\line(3,-4){1}}
        \put(32.7,48){\line(1,-2){.8}}
        \put(34,45){\line(1,-3){.5}}
        \put(34.7,42.3){\line(1,-6){.2}}
        \put(35,40){\line(0,-1){1}}
\end{picture}
    \end{minipage}\quad\!
    \text{$+$}\quad\quad\! 
    \begin{minipage}{50pt}
        \begin{picture}(50,60)
        \put(-9,24){$t$}
        \put(15,10){$-$}
        \put(15,40){$-$}
\put(20,0){\line(-1,0){1}}
        \put(17.6,0.2){\line(-4,1){1}}
        \put(15,0.9){\line(-3,1){1}}
        \put(12.5,2){\line(-2,1){1}}
        \put(10.6,3.2){\line(-1,1){.9}}
        \put(8.9,4.9){\line(-3,4){1}}
        \put(7.3,7){\line(-1,2){.8}}
        \put(6,10){\line(-1,3){.5}}
        \put(5.3,12.7){\line(-1,6){.2}}
        \put(5,15){\line(0,1){1}}
        \put(20,0){\line(1,0){1}}
        \put(22.5,0.2){\line(4,1){1}}
        \put(25,0.9){\line(3,1){1}}
        \put(27.5,2){\line(2,1){1}}
        \put(29.4,3.2){\line(1,1){.9}}
        \put(31.1,4.9){\line(3,4){1}}
        \put(32.7,7){\line(1,2){.8}}
        \put(34,10){\line(1,3){.5}}
        \put(34.7,12.7){\line(1,6){.2}}
        \put(35,15){\line(0,1){1}}
\qbezier(5,40)(20,20)(35,40)
\qbezier(5,15)(20,35)(35,15)
\put(20,55){\line(-1,0){1}}
        \put(17.6,54.8){\line(-4,-1){1}}
        \put(15,54.1){\line(-3,-1){1}}
        \put(12.5,53){\line(-2,-1){1}}
        \put(10.6,51.8){\line(-1,-1){.9}}
        \put(8.9,50.1){\line(-3,-4){1}}
        \put(7.3,48){\line(-1,-2){.8}}
        \put(6,45){\line(-1,-3){.5}}
        \put(5.3,42.3){\line(-1,-6){.2}}
        \put(5,40){\line(0,-1){1}}
        \put(20,55){\line(1,0){1}}
        \put(22.5,54.8){\line(4,-1){1}}
        \put(25,54.1){\line(3,-1){1}}
        \put(27.5,53){\line(2,-1){1}}
        \put(29.4,51.8){\line(1,-1){.9}}
        \put(31.1,50.1){\line(3,-4){1}}
        \put(32.7,48){\line(1,-2){.8}}
        \put(34,45){\line(1,-3){.5}}
        \put(34.7,42.3){\line(1,-6){.2}}
        \put(35,40){\line(0,-1){1}}
\end{picture}
    \end{minipage}\\
    \begin{minipage}{70pt}\label{f6}
        \begin{picture}(50,60)
        \put(20,25){$-$}
        \put(25,0){\line(-1,0){1}}
        \put(22.6,0.2){\line(-4,1){1}}
        \put(20,0.9){\line(-3,1){1}}
        \put(17.5,2){\line(-2,1){1}}
        \put(15.6,3.2){\line(-1,1){.9}}
        \put(13.9,4.9){\line(-3,4){1}}
        \put(12.3,7){\line(-1,2){.8}}
        \put(11,10){\line(-1,3){.5}}
        \put(10.3,12.7){\line(-1,6){.2}}
        \put(10,15){\line(0,1){1}}
        \put(25,0){\line(1,0){1}}
        \put(27.5,0.2){\line(4,1){1}}
        \put(30,0.9){\line(3,1){1}}
        \put(32.5,2){\line(2,1){1}}
        \put(34.4,3.2){\line(1,1){.9}}
        \put(36.1,4.9){\line(3,4){1}}
        \put(37.7,7){\line(1,2){.8}}
        \put(39,10){\line(1,3){.5}}
        \put(39.7,12.7){\line(1,6){.2}}
        \put(40,15){\line(0,1){1}}
    \qbezier(10,40)(25,25)(10,15)
\qbezier(40,15)(25,25)(40,40)
        \put(25,55){\line(-1,0){1}}
        \put(22.6,54.8){\line(-4,-1){1}}
        \put(20,54.1){\line(-3,-1){1}}
        \put(17.5,53){\line(-2,-1){1}}
        \put(15.6,51.8){\line(-1,-1){.9}}
        \put(13.9,50.1){\line(-3,-4){1}}
        \put(12.3,48){\line(-1,-2){.8}}
        \put(11,45){\line(-1,-3){.5}}
        \put(10.3,42.3){\line(-1,-6){.2}}
        \put(10,40){\line(0,-1){1}}
        \put(25,55){\line(1,0){1}}
        \put(27.5,54.8){\line(4,-1){1}}
        \put(30,54.1){\line(3,-1){1}}
        \put(32.5,53){\line(2,-1){1}}
        \put(34.4,51.8){\line(1,-1){.9}}
        \put(36.1,50.1){\line(3,-4){1}}
        \put(37.7,48){\line(1,-2){.8}}
        \put(39,45){\line(1,-3){.5}}
        \put(39.7,42.3){\line(1,-6){.2}}
        \put(40,40){\line(0,-1){1}}
        \end{picture}
    \end{minipage} 
&\begin{minipage}{80pt}
        \begin{picture}(70,60)
\put(0,24){\vector(1,0){50}}
\end{picture}
    \end{minipage}
\begin{minipage}{50pt}
        \begin{picture}(50,60)
        \put(15,10){$-$}
        \put(15,40){$+$}
\put(20,0){\line(-1,0){1}}
        \put(17.6,0.2){\line(-4,1){1}}
        \put(15,0.9){\line(-3,1){1}}
        \put(12.5,2){\line(-2,1){1}}
        \put(10.6,3.2){\line(-1,1){.9}}
        \put(8.9,4.9){\line(-3,4){1}}
        \put(7.3,7){\line(-1,2){.8}}
        \put(6,10){\line(-1,3){.5}}
        \put(5.3,12.7){\line(-1,6){.2}}
        \put(5,15){\line(0,1){1}}
        \put(20,0){\line(1,0){1}}
        \put(22.5,0.2){\line(4,1){1}}
        \put(25,0.9){\line(3,1){1}}
        \put(27.5,2){\line(2,1){1}}
        \put(29.4,3.2){\line(1,1){.9}}
        \put(31.1,4.9){\line(3,4){1}}
        \put(32.7,7){\line(1,2){.8}}
        \put(34,10){\line(1,3){.5}}
        \put(34.7,12.7){\line(1,6){.2}}
        \put(35,15){\line(0,1){1}}
\qbezier(5,40)(20,20)(35,40)
\qbezier(5,15)(20,35)(35,15)
\put(20,55){\line(-1,0){1}}
        \put(17.6,54.8){\line(-4,-1){1}}
        \put(15,54.1){\line(-3,-1){1}}
        \put(12.5,53){\line(-2,-1){1}}
        \put(10.6,51.8){\line(-1,-1){.9}}
        \put(8.9,50.1){\line(-3,-4){1}}
        \put(7.3,48){\line(-1,-2){.8}}
        \put(6,45){\line(-1,-3){.5}}
        \put(5.3,42.3){\line(-1,-6){.2}}
        \put(5,40){\line(0,-1){1}}
        \put(20,55){\line(1,0){1}}
        \put(22.5,54.8){\line(4,-1){1}}
        \put(25,54.1){\line(3,-1){1}}
        \put(27.5,53){\line(2,-1){1}}
        \put(29.4,51.8){\line(1,-1){.9}}
        \put(31.1,50.1){\line(3,-4){1}}
        \put(32.7,48){\line(1,-2){.8}}
        \put(34,45){\line(1,-3){.5}}
        \put(34.7,42.3){\line(1,-6){.2}}
        \put(35,40){\line(0,-1){1}}
\end{picture}
    \end{minipage}\quad\!
    \text{$+$}\quad\quad\! 
    \begin{minipage}{50pt}
        \begin{picture}(50,60)
        \put(15,10){$+$}
        \put(15,40){$-$}
\put(20,0){\line(-1,0){1}}
        \put(17.6,0.2){\line(-4,1){1}}
        \put(15,0.9){\line(-3,1){1}}
        \put(12.5,2){\line(-2,1){1}}
        \put(10.6,3.2){\line(-1,1){.9}}
        \put(8.9,4.9){\line(-3,4){1}}
        \put(7.3,7){\line(-1,2){.8}}
        \put(6,10){\line(-1,3){.5}}
        \put(5.3,12.7){\line(-1,6){.2}}
        \put(5,15){\line(0,1){1}}
        \put(20,0){\line(1,0){1}}
        \put(22.5,0.2){\line(4,1){1}}
        \put(25,0.9){\line(3,1){1}}
        \put(27.5,2){\line(2,1){1}}
        \put(29.4,3.2){\line(1,1){.9}}
        \put(31.1,4.9){\line(3,4){1}}
        \put(32.7,7){\line(1,2){.8}}
        \put(34,10){\line(1,3){.5}}
        \put(34.7,12.7){\line(1,6){.2}}
        \put(35,15){\line(0,1){1}}
\qbezier(5,40)(20,20)(35,40)
\qbezier(5,15)(20,35)(35,15)
\put(20,55){\line(-1,0){1}}
        \put(17.6,54.8){\line(-4,-1){1}}
        \put(15,54.1){\line(-3,-1){1}}
        \put(12.5,53){\line(-2,-1){1}}
        \put(10.6,51.8){\line(-1,-1){.9}}
        \put(8.9,50.1){\line(-3,-4){1}}
        \put(7.3,48){\line(-1,-2){.8}}
        \put(6,45){\line(-1,-3){.5}}
        \put(5.3,42.3){\line(-1,-6){.2}}
        \put(5,40){\line(0,-1){1}}
        \put(20,55){\line(1,0){1}}
        \put(22.5,54.8){\line(4,-1){1}}
        \put(25,54.1){\line(3,-1){1}}
        \put(27.5,53){\line(2,-1){1}}
        \put(29.4,51.8){\line(1,-1){.9}}
        \put(31.1,50.1){\line(3,-4){1}}
        \put(32.7,48){\line(1,-2){.8}}
        \put(34,45){\line(1,-3){.5}}
        \put(34.7,42.3){\line(1,-6){.2}}
        \put(35,40){\line(0,-1){1}}
\end{picture}
    \end{minipage}
    \text{$-$}\quad\quad\! 
    \begin{minipage}{50pt}
        \begin{picture}(50,60)
        \put(-9,24){$s$}
        \put(15,10){$-$}
        \put(15,40){$-$}
\put(20,0){\line(-1,0){1}}
        \put(17.6,0.2){\line(-4,1){1}}
        \put(15,0.9){\line(-3,1){1}}
        \put(12.5,2){\line(-2,1){1}}
        \put(10.6,3.2){\line(-1,1){.9}}
        \put(8.9,4.9){\line(-3,4){1}}
        \put(7.3,7){\line(-1,2){.8}}
        \put(6,10){\line(-1,3){.5}}
        \put(5.3,12.7){\line(-1,6){.2}}
        \put(5,15){\line(0,1){1}}
        \put(20,0){\line(1,0){1}}
        \put(22.5,0.2){\line(4,1){1}}
        \put(25,0.9){\line(3,1){1}}
        \put(27.5,2){\line(2,1){1}}
        \put(29.4,3.2){\line(1,1){.9}}
        \put(31.1,4.9){\line(3,4){1}}
        \put(32.7,7){\line(1,2){.8}}
        \put(34,10){\line(1,3){.5}}
        \put(34.7,12.7){\line(1,6){.2}}
        \put(35,15){\line(0,1){1}}
\qbezier(5,40)(20,20)(35,40)
\qbezier(5,15)(20,35)(35,15)
\put(20,55){\line(-1,0){1}}
        \put(17.6,54.8){\line(-4,-1){1}}
        \put(15,54.1){\line(-3,-1){1}}
        \put(12.5,53){\line(-2,-1){1}}
        \put(10.6,51.8){\line(-1,-1){.9}}
        \put(8.9,50.1){\line(-3,-4){1}}
        \put(7.3,48){\line(-1,-2){.8}}
        \put(6,45){\line(-1,-3){.5}}
        \put(5.3,42.3){\line(-1,-6){.2}}
        \put(5,40){\line(0,-1){1}}
        \put(20,55){\line(1,0){1}}
        \put(22.5,54.8){\line(4,-1){1}}
        \put(25,54.1){\line(3,-1){1}}
        \put(27.5,53){\line(2,-1){1}}
        \put(29.4,51.8){\line(1,-1){.9}}
        \put(31.1,50.1){\line(3,-4){1}}
        \put(32.7,48){\line(1,-2){.8}}
        \put(34,45){\line(1,-3){.5}}
        \put(34.7,42.3){\line(1,-6){.2}}
        \put(35,40){\line(0,-1){1}}
\end{picture}
    \end{minipage}\\
\begin{minipage}{70pt}\label{f7}
  \begin{picture}(90,60)
  \put(5,25){$p$}
            \qbezier(10,40)(25,25)(10,15)
\qbezier(40,15)(25,25)(40,40)
\put(40,25){$q$}
\end{picture}
    \end{minipage}
&\begin{minipage}{80pt}
\begin{picture}(80,60)
        \put(0,24){\vector(1,0){50}}
        \end{picture}
    \end{minipage}\quad
    \begin{minipage}{70pt}
  \begin{picture}(80,60)
  \put(10,42){$p:q$}
        \qbezier(5,40)(20,25)(35,40)
\qbezier(5,15)(20,30)(35,15)
\put(10,10){$q:p$}
\end{picture}
    \end{minipage}
\end{align}
\caption{A generalised Frobenius calculus of signed circles.  }\label{frobenius}
\end{figure}
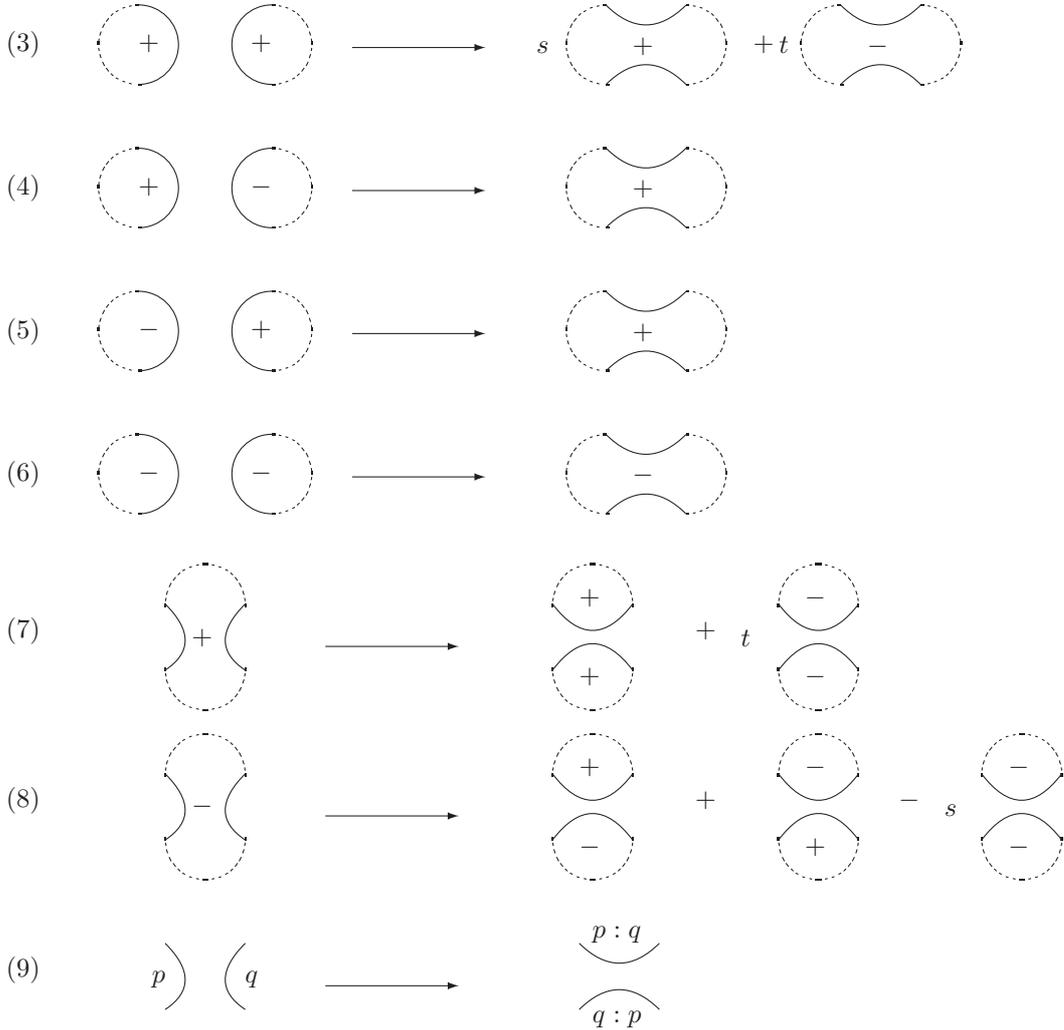

\section{Reidemeister invariance for the differential $\delta_{s, t}$.  }\label{reide}

\subsection{On the first Reidemeister invariance.  }\label{1st}
The first Reidemeister move is $D'$ $=$ $\begin{minipage}{30pt}
        \begin{picture}(30,30)
        \put(-2,13.5){$a$}
\qbezier(6.6,18)(0,25)(0,25)
\qbezier(0,5)(20,35)(25,18)
\qbezier(11,14)(20,5)(24.5,13)
\qbezier(24.5,13)(25.5,15.5)(25,18)
        \end{picture}
    \end{minipage}$ $\stackrel{1}{\sim}$ $\begin{minipage}{30pt}
        \begin{picture}(30,30)
        \cbezier(0,3)(30,5)(30,25)(0,27) 
        \end{picture}
    \end{minipage} $ $=$ $D$, we consider the composition 

\begin{equation}
\mathcal{C}(D') = \mathcal{C} \oplus \mathcal{C}_{\rm{contr}} \stackrel{\rho_{1}}{\to} \mathcal{C} \stackrel{\operatorname{isom}}{\to} \mathcal{C}(D)
\end{equation} 
where $a$ is a crossing and $\mathcal{C}$, $\mathcal{C}_{\rm{contr}}$, $\rho_{1}$ and the isomorphism are defined in the following formulas (\ref{1st-sum})--(\ref{1st-iso}).  

First, 
\begin{equation}\label{1st-sum}
\begin{split}
&\mathcal{C} := \mathcal{C} \left(~
    \begin{minipage}{30pt}
        \begin{picture}(30,30)
\qbezier(6.6,17)(0,24)(0,24)
\qbezier(0,4)(20,34)(25,17)
\qbezier(11,13)(20,4)(24.5,12)
\qbezier(24.5,12)(25.5,14.5)(25,17)
{\color{blue}{\put(8,19.5){\circle*{3}}
\put(8,10.5){\circle*{3}}
\put(8,11.5){\line(0,1){8}}}}
\put(-3.5,12.5){\text{$p$}}
\put(14,12.5){\text{$+$}}
        \end{picture}
    \end{minipage}\!\! \otimes [x] 
- \operatorname{m}(p:+) \begin{minipage}{30pt}
        \begin{picture}(30,30)
\qbezier(6.6,17)(0,24)(0,24)
\qbezier(0,4)(20,34)(25,17)
\qbezier(11,13)(20,4)(24.5,12)
\qbezier(24.5,12)(25.5,14.5)(25,17)
{\color{blue}{\put(8,19.5){\circle*{3}}
\put(8,10.5){\circle*{3}}
\put(8,11.5){\line(0,1){8}}}}
\put(14,12.5){\text{$-$}}
        \end{picture}
    \end{minipage} \!\! \otimes [x]
   \right), \\
   &\mathcal{C}_{\rm{contr}} := \mathcal{C}\left(~ 
\begin{minipage}{30pt}
        \begin{picture}(30,30)
\qbezier(6.6,17)(0,24)(0,24)
\qbezier(0,4)(20,34)(25,17)
\qbezier(11,13)(20,4)(24.5,12)
\qbezier(24.5,12)(25.5,14.5)(25,17)
{\color{blue}{\put(8,19.5){\circle*{3}}
\put(8,10.5){\circle*{3}}
\put(8,11.5){\line(0,1){8}}}}
\put(-3.5,12.5){\text{$p$}}
\put(14,12.5){\text{$-$}}
        \end{picture}
    \end{minipage} \!\! \otimes [x],   
    \begin{minipage}{30pt}
        \begin{picture}(30,30)
\qbezier(6.6,17)(0,24)(0,24)
\qbezier(0,4)(20,34)(25,17)
\qbezier(11,13)(20,4)(24.5,12)
\qbezier(24.5,12)(25.5,14.5)(25,17)
{\color{red}{\put(12.5,15.5){\circle*{3}}
\put(4.5,15.5){\circle*{3}}
\put(4.5,15.5){\line(1,0){8}}}}
\put(15.5,14){\text{$p$}}
        \end{picture}
    \end{minipage} \!\! \otimes [xa] \right).  
\end{split}
\end{equation}

Second, the retraction 
$\rho_{1} : \mathcal{C}\left(~
    \begin{minipage}{30pt}
        \begin{picture}(30,30)
\qbezier(6.6,18)(0,25)(0,25)
\qbezier(0,5)(20,35)(25,18)
\qbezier(11,14)(20,5)(24.5,13)
\qbezier(24.5,13)(25.5,15.5)(25,18)
        \end{picture}
    \end{minipage}
\right)$ $\to$ 
$ 
\mathcal{C}\left(~\begin{minipage}{30pt}
        \begin{picture}(30,30)
\qbezier(6.6,17)(0,24)(0,24)
\qbezier(0,4)(20,34)(25,17)
\qbezier(11,13)(20,4)(24.5,12)
\qbezier(24.5,12)(25.5,14.5)(25,17)
{\color{blue}{\put(8,19.5){\circle*{3}}
\put(8,10.5){\circle*{3}}
\put(8,11.5){\line(0,1){8}}}}
\put(-3.5,12.5){\text{$p$}}
\put(14,12.5){\text{$+$}}
        \end{picture}
    \end{minipage}\!\! \otimes [x] 
- \operatorname{m}(p:+) \begin{minipage}{30pt}
        \begin{picture}(30,30)
\qbezier(6.6,17)(0,24)(0,24)
\qbezier(0,4)(20,34)(25,17)
\qbezier(11,13)(20,4)(24.5,12)
\qbezier(24.5,12)(25.5,14.5)(25,17)
{\color{blue}{\put(8,19.5){\circle*{3}}
\put(8,10.5){\circle*{3}}
\put(8,11.5){\line(0,1){8}}}}
\put(14,12.5){\text{$-$}}
        \end{picture}
    \end{minipage} \!\! \otimes [x]
    \right)$

is defined by the formulas

\begin{equation}\label{1st-ret}
\begin{split}
&\begin{minipage}{30pt}
        \begin{picture}(30,30)
        \qbezier(6.6,17)(0,24)(0,24)
\qbezier(0,4)(20,34)(25,17)
\qbezier(11,13)(20,4)(24.5,12)
\qbezier(24.5,12)(25.5,14.5)(25,17)
{\color{blue}{\put(8,19.5){\circle*{3}}
\put(8,10.5){\circle*{3}}
\put(8,11.5){\line(0,1){8}}}}
\put(-3.5,12.5){\text{$p$}}
\put(14,12.5){\text{$+$}}
        \end{picture}
    \end{minipage} \!\! \otimes [x] \mapsto \ 
\begin{minipage}{30pt}
        \begin{picture}(30,30)
\qbezier(6.6,17)(0,24)(0,24)
\qbezier(0,4)(20,34)(25,17)
\qbezier(11,13)(20,4)(24.5,12)
\qbezier(24.5,12)(25.5,14.5)(25,17)
{\color{blue}{\put(8,19.5){\circle*{3}}
\put(8,10.5){\circle*{3}}
\put(8,11.5){\line(0,1){8}}}}
\put(-3.5,12.5){\text{$p$}}
\put(14,12.5){\text{$+$}}
        \end{picture}
    \end{minipage}\!\! \otimes [x] 
- \operatorname{m}(p:+) \begin{minipage}{30pt}
        \begin{picture}(30,30)
\qbezier(6.6,17)(0,24)(0,24)
\qbezier(0,4)(20,34)(25,17)
\qbezier(11,13)(20,4)(24.5,12)
\qbezier(24.5,12)(25.5,14.5)(25,17)
{\color{blue}{\put(8,19.5){\circle*{3}}
\put(8,10.5){\circle*{3}}
\put(8,11.5){\line(0,1){8}}}}
\put(14,12.5){\text{$-$}}
        \end{picture}
    \end{minipage} \!\! \otimes [x], \\
&    \begin{minipage}{30pt}
        \begin{picture}(30,30)
\qbezier(6.6,17)(0,24)(0,24)
\qbezier(0,4)(20,34)(25,17)
\qbezier(11,13)(20,4)(24.5,12)
\qbezier(24.5,12)(25.5,14.5)(25,17)
{\color{blue}{\put(8,19.5){\circle*{3}}
\put(8,10.5){\circle*{3}}
\put(8,11.5){\line(0,1){8}}}}
\put(-3.5,12.5){\text{$p$}}
\put(14,12.5){\text{$-$}}
        \end{picture}
    \end{minipage} \!\! \otimes [x],~\   
    \begin{minipage}{30pt}
        \begin{picture}(30,30)
\qbezier(6.6,17)(0,24)(0,24)
\qbezier(0,4)(20,34)(25,17)
\qbezier(11,13)(20,4)(24.5,12)
\qbezier(24.5,12)(25.5,14.5)(25,17)
{\color{red}{\put(12.5,15.5){\circle*{3}}
\put(4.5,15.5){\circle*{3}}
\put(4.5,15.5){\line(1,0){8}}}}
\put(16,14){\text{$p$}}
        \end{picture}
    \end{minipage} \!\! \otimes [xa] \mapsto \ 0.  
\end{split}
\end{equation}
It is easy to see that $\delta_{s, t} \circ \rho_{1}$ $=$ $\rho_{1} \circ \delta_{s, t}$.  Then $\rho_{1}$ is certainly a chain map.  Note that we have $\delta_{s, t} \circ \rho_{1}$ $=$ $\rho_{1} \circ \delta_{s, t}$ using only (\ref{dif}) and not using the property (\ref{f1})--(\ref{f6}).

Third, the isomorphism 
\begin{equation}\label{1st-iso1}
\begin{split}
&\mathcal{C}\left(~
    \begin{minipage}{30pt}
        \begin{picture}(30,30)
\qbezier(6.6,17)(0,24)(0,24)
\qbezier(0,4)(20,34)(25,17)
\qbezier(11,13)(20,4)(24.5,12)
\qbezier(24.5,12)(25.5,14.5)(25,17)
{\color{blue}{\put(8,19.5){\circle*{3}}
\put(8,10.5){\circle*{3}}
\put(8,11.5){\line(0,1){8}}}}
\put(-3.5,12.5){\text{$p$}}
\put(14,12.5){\text{$+$}}
        \end{picture}
    \end{minipage}\!\! \otimes [x] 
- \operatorname{m}(p:+) \begin{minipage}{30pt}
        \begin{picture}(30,30)
\qbezier(6.6,17)(0,24)(0,24)
\qbezier(0,4)(20,34)(25,17)
\qbezier(11,13)(20,4)(24.5,12)
\qbezier(24.5,12)(25.5,14.5)(25,17)
{\color{blue}{\put(8,19.5){\circle*{3}}
\put(8,10.5){\circle*{3}}
\put(8,11.5){\line(0,1){8}}}}
\put(14,12.5){\text{$-$}}
        \end{picture}
    \end{minipage} \!\! \otimes [x] \right) 
    \to 
    \mathcal{C}\left(~
    \begin{minipage}{30pt}
        \begin{picture}(30,30)
        \cbezier(0,3)(30,5)(30,25)(0,27) 
        \end{picture}
    \end{minipage} \!\! \otimes [x]
    \right) 
    \end{split}
        \end{equation}
    
    is defined by the formulas 
    
\begin{equation}\label{1st-iso}
\begin{split}
&\begin{minipage}{30pt}
        \begin{picture}(30,30)
\qbezier(6.6,17)(0,24)(0,24)
\qbezier(0,4)(20,34)(25,17)
\qbezier(11,13)(20,4)(24.5,12)
\qbezier(24.5,12)(25.5,14.5)(25,17)
{\color{blue}{\put(8,19.5){\circle*{3}}
\put(8,10.5){\circle*{3}}
\put(8,11.5){\line(0,1){8}}}}
\put(-3.5,12.5){\text{$p$}}
\put(14,12.5){\text{$+$}}
        \end{picture}
    \end{minipage}\!\! \otimes [x] 
- \operatorname{m}(p:+) \begin{minipage}{30pt}
        \begin{picture}(30,30)
\qbezier(6.6,17)(0,24)(0,24)
\qbezier(0,4)(20,34)(25,17)
\qbezier(11,13)(20,4)(24.5,12)
\qbezier(24.5,12)(25.5,14.5)(25,17)
{\color{blue}{\put(8,19.5){\circle*{3}}
\put(8,10.5){\circle*{3}}
\put(8,11.5){\line(0,1){8}}}}
\put(14,12.5){\text{$-$}}
        \end{picture}
    \end{minipage} \!\! \otimes [x] \mapsto~\ 
    \begin{minipage}{30pt}
        \begin{picture}(30,30)
        \cbezier(0,3)(30,5)(30,25)(0,27) 
        \put(5,12.5){\text{$p$}}
        \end{picture}
    \end{minipage} \!\! \otimes [x].  
        \end{split}
        \end{equation}  
The homotopy connecting $\operatorname{in}$ $\circ$ $\rho_{1}$ to the identity $:$ $\mathcal{C}\left(~
    \begin{minipage}{30pt}
        \begin{picture}(30,30)
\qbezier(6.6,18)(0,25)(0,25)
\qbezier(0,5)(20,35)(25,18)
\qbezier(11,14)(20,5)(24.5,13)
\qbezier(24.5,13)(25.5,15.5)(25,18)
        \end{picture}
    \end{minipage}
\right)$ $\to$ 
$\mathcal{C}\left(~
    \begin{minipage}{30pt}
        \begin{picture}(30,30)
\qbezier(6.6,18)(0,25)(0,25)
\qbezier(0,5)(20,35)(25,18)
\qbezier(11,14)(20,5)(24.5,13)
\qbezier(24.5,13)(25.5,15.5)(25,18)
        \end{picture}
    \end{minipage}
\right)$ such that $\delta_{s, t}$ $\circ$ $h_{1}$ $+$  $h_{1}$ $\circ$ $\delta_{s, t}$ $=$ $\operatorname{id} - \operatorname{in} \circ \rho_{1}$, is defined by the formulas: 
\begin{equation}\label{first-hom}
\begin{split}
\begin{minipage}{30pt}
        \begin{picture}(30,30)
\qbezier(6.6,17)(0,24)(0,24)
\qbezier(0,4)(20,34)(25,17)
\qbezier(11,13)(20,4)(24.5,12)
\qbezier(24.5,12)(25.5,14.5)(25,17)
{\color{red}{\put(12.5,15.5){\circle*{3}}
\put(4.5,15.5){\circle*{3}}
\put(4.5,15.5){\line(1,0){8}}}}
\put(16,14){\text{$p$}}
        \end{picture}
    \end{minipage} \!\! \otimes [xa] &\mapsto~ 
    \begin{minipage}{30pt}
        \begin{picture}(30,30)
\qbezier(6.6,17)(0,24)(0,24)
\qbezier(0,4)(20,34)(25,17)
\qbezier(11,13)(20,4)(24.5,12)
\qbezier(24.5,12)(25.5,14.5)(25,17)
{\color{blue}{\put(8,19.5){\circle*{3}}
\put(8,10.5){\circle*{3}}
\put(8,11.5){\line(0,1){8}}}}
\put(-3.5,12.5){\text{$p$}}
\put(14,12.5){\text{$-$}}
        \end{picture}
    \end{minipage} \!\! \otimes [x],~
    {\text{otherwise}} \mapsto 0.  
\end{split}
\end{equation}
\begin{rem}
The explicit formula (\ref{first-hom}) of the homotopy map $h_{1}$ in the case ($s$ $=$ $t$ $=$ $0$) of the original Khovanov homology is given by Oleg Viro \cite[Subsection 5.5]{viro}.  
\end{rem}
We can verify $\delta_{s, t}$ $\circ$ $h_{1}$ $+$  $h_{1}$ $\circ$ $\delta_{s, t}$ $=$ $\operatorname{id} - \operatorname{in} \circ \rho_{1}$ by a direct computation as follows.  
\begin{equation}\label{first-begin}
\begin{split}
\left(h_{1} \circ \delta_{s, t} + \delta_{s, t} \circ h_{1} \right)\left(
\begin{minipage}{30pt}
        \begin{picture}(30,30)
\qbezier(6.6,17)(0,24)(0,24)
\qbezier(0,4)(20,34)(25,17)
\qbezier(11,13)(20,4)(24.5,12)
\qbezier(24.5,12)(25.5,14.5)(25,17)
{\color{blue}{\put(8,19.5){\circle*{3}}
\put(8,10.5){\circle*{3}}
\put(8,11.5){\line(0,1){8}}}}
\put(-3.5,12.5){\text{$p$}}
\put(14,12.5){\text{$+$}}
        \end{picture}
    \end{minipage} \!\! \otimes [x]\right) 
    &= h_{1}\left( \operatorname{m}(p:+) \begin{minipage}{30pt}
        \begin{picture}(30,30)
\qbezier(6.6,17)(0,24)(0,24)
\qbezier(0,4)(20,34)(25,17)
\qbezier(11,13)(20,4)(24.5,12)
\qbezier(24.5,12)(25.5,14.5)(25,17)
{\color{red}{\put(12.5,15.5){\circle*{3}}
\put(4.5,15.5){\circle*{3}}
\put(4.5,15.5){\line(1,0){8}}}}
        \end{picture}
    \end{minipage} \!\! \otimes [xa]
      \right)\\
    &= \operatorname{m}(p:+) \begin{minipage}{30pt}
        \begin{picture}(30,30)
\qbezier(6.6,17)(0,24)(0,24)
\qbezier(0,4)(20,34)(25,17)
\qbezier(11,13)(20,4)(24.5,12)
\qbezier(24.5,12)(25.5,14.5)(25,17)
{\color{blue}{\put(8,19.5){\circle*{3}}
\put(8,10.5){\circle*{3}}
\put(8,11.5){\line(0,1){8}}}}
\put(14,12.5){\text{$-$}}
        \end{picture}
    \end{minipage} \!\! \otimes [x]\\
    &= \left( \operatorname{id} - \rho_{1} \right)\left(
    \begin{minipage}{30pt}
        \begin{picture}(30,30)
\qbezier(6.6,17)(0,24)(0,24)
\qbezier(0,4)(20,34)(25,17)
\qbezier(11,13)(20,4)(24.5,12)
\qbezier(24.5,12)(25.5,14.5)(25,17)
{\color{blue}{\put(8,19.5){\circle*{3}}
\put(8,10.5){\circle*{3}}
\put(8,11.5){\line(0,1){8}}}}
\put(-3.5,12.5){\text{$p$}}
\put(14,12.5){\text{$+$}}
        \end{picture}
    \end{minipage} \!\! \otimes [x]
    \right).  
    \end{split}
    \end{equation}
    Similarly, 
  \begin{equation}
\begin{split}
\left(h_{1} \circ \delta_{s, t} + \delta_{s, t} \circ h_{1} \right)\left(
\begin{minipage}{30pt}
        \begin{picture}(30,30)
\qbezier(6.6,17)(0,24)(0,24)
\qbezier(0,4)(20,34)(25,17)
\qbezier(11,13)(20,4)(24.5,12)
\qbezier(24.5,12)(25.5,14.5)(25,17)
{\color{blue}{\put(8,19.5){\circle*{3}}
\put(8,10.5){\circle*{3}}
\put(8,11.5){\line(0,1){8}}}}
\put(-3.5,12.5){\text{$p$}}
\put(14,12.5){\text{$-$}}
        \end{picture}
    \end{minipage} \!\! \otimes [x]\right) 
    &= h_{1}\left( \begin{minipage}{30pt}
        \begin{picture}(30,30)
\qbezier(6.6,17)(0,24)(0,24)
\qbezier(0,4)(20,34)(25,17)
\qbezier(11,13)(20,4)(24.5,12)
\qbezier(24.5,12)(25.5,14.5)(25,17)
{\color{red}{\put(12.5,15.5){\circle*{3}}
\put(4.5,15.5){\circle*{3}}
\put(4.5,15.5){\line(1,0){8}}}}
\put(16,14){\text{$p$}}
        \end{picture}
    \end{minipage} \!\! \otimes [xa] \right)\\
    &= \quad \begin{minipage}{30pt}
        \begin{picture}(30,30)
\qbezier(6.6,17)(0,24)(0,24)
\qbezier(0,4)(20,34)(25,17)
\qbezier(11,13)(20,4)(24.5,12)
\qbezier(24.5,12)(25.5,14.5)(25,17)
{\color{blue}{\put(8,19.5){\circle*{3}}
\put(8,10.5){\circle*{3}}
\put(8,11.5){\line(0,1){8}}}}
\put(-3.5,12.5){\text{$p$}}
\put(14,12.5){\text{$-$}}
        \end{picture}
    \end{minipage} \!\! \otimes [x]\\
    &= \left( \operatorname{id} - \rho_{1} \right)\left(
    \begin{minipage}{30pt}
        \begin{picture}(30,30)
\qbezier(6.6,17)(0,24)(0,24)
\qbezier(0,4)(20,34)(25,17)
\qbezier(11,13)(20,4)(24.5,12)
\qbezier(24.5,12)(25.5,14.5)(25,17)
{\color{blue}{\put(8,19.5){\circle*{3}}
\put(8,10.5){\circle*{3}}
\put(8,11.5){\line(0,1){8}}}}
\put(-3.5,12.5){\text{$p$}}
\put(14,12.5){\text{$-$}}
        \end{picture}
    \end{minipage} \!\! \otimes [x]
    \right),   
    \end{split}
    \end{equation}
        \begin{equation}\label{1st-last}
\left(h_{1} \circ \delta_{s, t} + \delta_{s, t} \circ h_{1} \right)\left(\begin{minipage}{30pt}
        \begin{picture}(30,30)
\qbezier(6.6,17)(0,24)(0,24)
\qbezier(0,4)(20,34)(25,17)
\qbezier(11,13)(20,4)(24.5,12)
\qbezier(24.5,12)(25.5,14.5)(25,17)
{\color{red}{\put(12.5,15.5){\circle*{3}}
\put(4.5,15.5){\circle*{3}}
\put(4.5,15.5){\line(1,0){8}}}}
\put(16,14){\text{$p$}}
        \end{picture}
    \end{minipage} \!\! \otimes [xa]
\right) 
    = \begin{minipage}{30pt}
        \begin{picture}(30,30)
\qbezier(6.6,17)(0,24)(0,24)
\qbezier(0,4)(20,34)(25,17)
\qbezier(11,13)(20,4)(24.5,12)
\qbezier(24.5,12)(25.5,14.5)(25,17)
{\color{red}{\put(12.5,15.5){\circle*{3}}
\put(4.5,15.5){\circle*{3}}
\put(4.5,15.5){\line(1,0){8}}}}
\put(16,14){\text{$p$}}
        \end{picture}
    \end{minipage} \!\! \otimes [xa]
    = \left( \operatorname{id} - \rho_{1} \right)\left(
    \begin{minipage}{30pt}
        \begin{picture}(30,30)
\qbezier(6.6,17)(0,24)(0,24)
\qbezier(0,4)(20,34)(25,17)
\qbezier(11,13)(20,4)(24.5,12)
\qbezier(24.5,12)(25.5,14.5)(25,17)
{\color{red}{\put(12.5,15.5){\circle*{3}}
\put(4.5,15.5){\circle*{3}}
\put(4.5,15.5){\line(1,0){8}}}}
\put(16,14){\text{$p$}}
        \end{picture}
    \end{minipage} \!\! \otimes [xa]
\right).  
\end{equation}

\subsection{On the second Reidemeister invariance.  }\label{2nd}
In this section,  consider the second Reidemeister move of Khovanov homology for $\delta_{s,t}$.  
In \cite[Subsection 2.2]{ito}, a chain homotopy map, a retraction and an isomorphism for the second Reidemeister invariance of Khovanov homology for $\delta_{0, 0}$ were given.  

We extend the chain homotopy map, the retraction and the isomorphism for $\delta_{0, 0}$ to maps $h_{2}$, $\rho_{2}$ and $\operatorname{isom}_{2}$ for $\delta_{s, t}$ by replacing $p:q$ and $q:p$ defined \cite[Figure 3]{jacobsson} with the ``generalised" $p:q$ and $q:p$ defined Figure \ref{frobenius}.  We will show that these extended maps become the chain homotopy map, the retraction and the isomorphism for the second Reidemeister invariance of Khovanov homology for $\delta_{s, t}$.  

Let $a$, $b$ be crossings, $x$ sequence of crossings with negative markers and $p$, $q$ be signs.  For a crossing with no markers or no signs in the following formulas, any markers or signs are allowed.  Let $\tilde{p}$ be $p$ unless the upper arc is connected to one of the other arcs in the picture.  

For a link diagram $D'$ $=$ $\begin{minipage}{30pt}
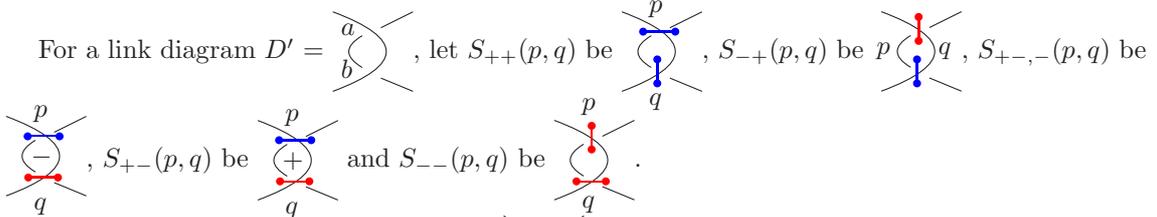
        \begin{picture}(30,30)
            \qbezier(0,0)(40,15)(0,30)
            \qbezier(11,21)(1,15)(11,9)
            \qbezier(18,24)(24,27)(30,30)
            \qbezier(18,5)(24,2.5)(30,0)
            \put(3,21){$a$}
            \put(3,5){$b$}
        \end{picture}
    \end{minipage}
$, let $S_{++}(p,q)$ be $\begin{minipage}{30pt}
        \begin{picture}(30,30)
            \qbezier(0,0)(40,15)(0,30)
            \qbezier(11,21)(1,15)(11,9)
            \qbezier(18,24)(24,27)(30,30)
            \qbezier(18,5)(24,2.5)(30,0)
            {\color{blue}
            {\put(8,22.5){\circle*{3}}
\put(20,22.5){\circle*{3}}
\put(9,22.5){\line(1,0){10}}
            }}
            {\color{blue}{\put(10,12){\circle*{3}}
\put(10,3){\circle*{3}}
\put(10,4){\line(0,1){8}}}}
\put(7,30){$p$}
\put(7,-5){$q$}
        \end{picture}
    \end{minipage}$, $S_{-+}(p,q)$ be\ \ 
    $\begin{minipage}{30pt}
        \begin{picture}(40,40)
            \qbezier(0,5)(40,20)(0,35)
            \qbezier(11,26)(1,20)(11,14)
            \qbezier(18,29)(24,32)(30,35)
            \qbezier(18,10)(24,7.5)(30,5)
  {\color{red}{\put(14,33){\circle*{3}}
\put(14,24){\circle*{3}}
\put(14,24){\line(0,1){8}}}}
            {\color{blue}{\put(10,17){\circle*{3}}
\put(10,8){\circle*{3}}
\put(10,9){\line(0,1){8}}}}
\put(-5,19){$p$}
\put(18,19){$q$}
        \end{picture} 
        \end{minipage}$, $S_{+-,-}(p,q)$ be 
        $\begin{minipage}{30pt}
        \begin{picture}(30,37)
            \qbezier(0,5)(40,20)(0,35)
            \qbezier(11,26)(1,20)(11,14)
            \qbezier(18,29)(24,32)(30,35)
            \qbezier(18,10)(24,7.5)(30,5)
            {\color{blue}
            {\put(8,27.5){\circle*{3}}
\put(20,27.5){\circle*{3}}
\put(9,27.5){\line(1,0){10}}
            }}
            {\color{red}{\put(16,12){\circle*{3}}
\put(5,12){\circle*{3}}
\put(5,12){\line(1,0){10}}}}
\put(6,17){\text{$-$}}
\put(7,35){$p$}
\put(7,0){$q$}
        \end{picture}
    \end{minipage}$, $S_{+-}(p, q)$ be 
    $\begin{minipage}{30pt}
        \begin{picture}(30,40)
            \qbezier(0,5)(40,20)(0,35)
            \qbezier(11,26)(1,20)(11,14)
            \qbezier(18,29)(24,32)(30,35)
            \qbezier(18,10)(24,7.5)(30,5)
            {\color{blue}
            {\put(8,27.5){\circle*{3}}
\put(20,27.5){\circle*{3}}
\put(9,27.5){\line(1,0){10}}
            }}
            {\color{red}{\put(16,12){\circle*{3}}
\put(5,12){\circle*{3}}
\put(5,12){\line(1,0){10}}}}
\put(6,17){\text{$+$}}
\put(7,35){$p$}
\put(7,0){$q$}
        \end{picture}
    \end{minipage}$ and $S_{--}(p,q)$ be 
    $\begin{minipage}{30pt}
        \begin{picture}(30,30)
            \qbezier(0,0)(40,15)(0,30)
            \qbezier(11,21)(1,15)(11,9)
            \qbezier(18,24)(24,27)(30,30)
            \qbezier(18,5)(24,2.5)(30,0)
  {\color{red}{\put(14,28){\circle*{3}}
\put(14,19){\circle*{3}}
\put(14,19){\line(0,1){8}}}}
            {\color{red}{\put(16,7){\circle*{3}}
\put(5,7){\circle*{3}}
\put(5,7){\line(1,0){10}}}}
\put(7,34){$p$}
\put(7,-3){$q$}
        \end{picture}
    \end{minipage}$.  

The second Reidemeister move is $D'$ $\stackrel{2}{\sim}$ $\begin{minipage}{30pt}
        \begin{picture}(30,30)
            \qbezier(0,0)(20,15)(0,30)
            \qbezier(30,0)(10,15)(30,30)
        \end{picture}
    \end{minipage}$ $=$ $D$, we consider the composition 

\begin{equation}
\mathcal{C}(D') = \mathcal{C} \oplus \mathcal{C}_{\rm{contr}} \stackrel{\rho_{2}}{\to} \mathcal{C} \stackrel{\operatorname{isom_{2}}}{\to} \mathcal{C}(D)
\end{equation} 
where $\mathcal{C}$ $\oplus$ $\mathcal{C}_{\rm{contr}}$, $\rho_{2}$ and $\operatorname{isom}_{2}$ are defined in \cite[Section 2.1]{ito} by replacing $p:q$ and $q:p$ in \cite[Section 2.1]{ito} with the ``generalised" $p:q$ and $q:p$ defined Figure \ref{frobenius}.  

Let $h_{2}$ be the chain homotopy maps given by \cite[Section 2.1, Equation (7)]{ito}.  To verify the second Reidemeister invariance of $\mathcal{H}^{i}(D)$, it is sufficient to show that $\delta_{s, t} \circ \rho_{2}$ $=$ $\rho_{2} \circ \delta_{s, t}$ and $\delta_{s, t}$ $\circ$ $h_{2}$ $+$  $h_{2}$ $\circ$ $\delta_{s, t}$ $=$ $\operatorname{id} - \operatorname{in} \circ \rho_{2}$.  In the following, we reveal where are non-trivial parts in the proof of these two equation and figure out a good way to avoid the ``hard" calculation.  

$\bullet$ On the proof of the $\delta_{s, t} \circ \rho_{2}$ $=$ $\rho_{2} \circ \delta_{s, t}$.  

We denote a big sum of the states $\sum_{u} S \otimes [xu]$ by $\mathcal{S}^{u}$ where $u$ ranges over crossings of $D'$ except for $a$ and $b$.   

First, it is easy to see that we have $\delta_{s, t} \circ \rho_{2}(S_{+-,-}(p, q))$ $=$ $(\rho_{2} \circ \delta_{s, t})(S_{+-, -}(p, q))$ $=$ $0$ depending only (\ref{dif}); $\delta_{s, t} \circ \rho_{2}(S_{+-}(p, q))$ $=$ $(\rho_{2} \circ \delta_{s, t})(S_{+-}(p, q))$ $=$ $\rho_{2}(\mathcal{S}^{u}_{-+})$ and $\delta_{s, t} \circ \rho_{2}(S_{-+}(p, q))$ $=$ $(\rho_{2} \circ \delta_{s, t})(S_{-+}(p, q))$ $=$ $\rho_{2}(\mathcal{S}^{u}_{-+})$ follows from (\ref{f2})--(\ref{f4}).  

Second, 
\begin{equation}\label{2nd-ret}
\begin{split}
(\rho_{2} \circ \delta_{s, t})(S_{++}(p, q) \otimes [x])
&= \rho_{2}\left(S_{-+}(p:q, q:p) \otimes [xa] + \qquad\begin{minipage}{30pt}
        \begin{picture}(30,37)
            \qbezier(0,5)(40,20)(0,35)
            \qbezier(11,26)(1,20)(11,14)
            \qbezier(18,29)(24,32)(30,35)
            \qbezier(18,10)(24,7.5)(30,5)
            {\color{blue}
            {\put(8,27.5){\circle*{3}}
\put(20,27.5){\circle*{3}}
\put(9,27.5){\line(1,0){10}}
            }}
            {\color{red}{\put(16,12){\circle*{3}}
\put(5,12){\circle*{3}}
\put(5,12){\line(1,0){10}}}}
\put(-25,17){$q:q$}
\qbezier(-3,21)(2,24)(9,19)
\put(7,35){$p$}
\put(0,0){$q:q$}
        \end{picture}
    \end{minipage} \!\! \otimes [xb]\right)\\
&= - \rho_{2}(S_{+-}(p, q) \otimes [xb]) + \rho_{2}(S_{+-}(p, q) \otimes [xb])\\
&= 0.  
\end{split}
\end{equation}

The second equality of (\ref{2nd-ret}) follows from 
\begin{equation}\label{1to1eq}
f_{+-}(\delta_{s,t}(S_{++}(p,q) \otimes [x])) = S_{+-}(p,q) \otimes [xb]
\end{equation}
where $f_{+-}$ is the homomorphism defined by $f_{+-}(S)$ $=$ $S$ if $S$ $=$ $S_{+-}(p, q) \otimes [xb]$ of any $p$, $q$, $f_{+-}(S)$ $=$ $0$ otherwise.  
We have (\ref{1to1eq}) because $(\ref{f4})$ and (\ref{f5}) implies that $\delta_{s,t}(S_{++}(p,q) \otimes [x])$ $=$ $S_{+-}(p,q) \otimes [xb]$ $+$ $\sum_{u} S_{u}$ where $S_{u}$ is neither $S_{+-}(p,q) \otimes [xb]$ of any $p$ and $q$ nor $S_{--}(p, q) \otimes [xab]$ of any $p$ and $q$.  

Then $\rho_{2}$ is certainly a chain map.  

$\bullet$ On the proof of the $\delta_{s, t}$ $\circ$ $h_{2}$ $+$  $h_{2}$ $\circ$ $\delta_{s, t}$ $=$ $\operatorname{id} - \operatorname{in} \circ \rho_{2}$.  

In the beginning, let us show the equation depending only on (\ref{f2})--(\ref{f4}) of Figure \ref{frobenius}.  

\begin{equation}
\begin{split}
\left(h_{2} \circ \delta_{s, t} + \delta_{s, t} \circ h_{2} \right) (S_{--}(p, q) \otimes [xab]) &= S_{--}(p,q) \otimes [xab]\\
&=(\operatorname{id} - \rho_{2})(S_{--}(p, q) \otimes [xab]), 
\end{split}
\end{equation}
\begin{equation}
\begin{split}
\left(h_{2} \circ \delta_{s, t} + \delta_{s, t} \circ h_{2} \right) (S_{-+}(p, q) \otimes [xa]) &=h_{2}(S_{--}(p:q, q:p) \otimes [xab])\\
&=- S_{+-,-}(p:q, q:p) \otimes [xb]\\
&=(\operatorname{id} - \rho_{2})(S_{-+}(p, q) \otimes [xa]), 
\end{split}
\end{equation}
\begin{equation}
\begin{split}
\left(h_{2} \circ \delta_{s, t} + \delta_{s, t} \circ h_{2} \right) (S_{+-,-}(p, q) \otimes [xb]) &=h_{2}(-S_{--}(p,q) \otimes [xab])\\
&= S_{+-,-}(p,q) \otimes [xb]\\
&=(\operatorname{id} - \rho_{2})(S_{+-,-}(p, q) \otimes [xb]).  
\end{split}
\end{equation}

Second, let us show the equation depending only on (\ref{f5})--(\ref{f6}) of Figure \ref{frobenius}.  

\begin{equation}\label{2nd_c1}
\begin{split}
\left(h_{2} \circ \delta_{s, t} + \delta_{s, t} \circ h_{2} \right) (S_{++}(p, q) \otimes [x]) &=h_{2}\left(\qquad\begin{minipage}{30pt}
        \begin{picture}(30,37)
            \qbezier(0,5)(40,20)(0,35)
            \qbezier(11,26)(1,20)(11,14)
            \qbezier(18,29)(24,32)(30,35)
            \qbezier(18,10)(24,7.5)(30,5)
            {\color{blue}
            {\put(8,27.5){\circle*{3}}
\put(20,27.5){\circle*{3}}
\put(9,27.5){\line(1,0){10}}
            }}
            {\color{red}{\put(16,12){\circle*{3}}
\put(5,12){\circle*{3}}
\put(5,12){\line(1,0){10}}}}
\put(-25,17){$q:q$}
\qbezier(-3,21)(2,24)(9,19)
\put(7,35){$p$}
\put(0,0){$q:q$}
        \end{picture}
    \end{minipage} \!\! \otimes [xb] \right)\\
&= S_{++}(p,q) \otimes [x]\\
&=(\operatorname{id} - \rho_{2})(S_{++}(p, q) \otimes [x]).  
\end{split}
\end{equation}

The first equality of (\ref{2nd_c1}) follows from Lemma \ref{1to1} 
\begin{lem}\label{1to1}
$h_{2}\left(\qquad\begin{minipage}{30pt}
        \begin{picture}(30,37)
            \qbezier(0,5)(40,20)(0,35)
            \qbezier(11,26)(1,20)(11,14)
            \qbezier(18,29)(24,32)(30,35)
            \qbezier(18,10)(24,7.5)(30,5)
            {\color{blue}
            {\put(8,27.5){\circle*{3}}
\put(20,27.5){\circle*{3}}
\put(9,27.5){\line(1,0){10}}
            }}
            {\color{red}{\put(16,12){\circle*{3}}
\put(5,12){\circle*{3}}
\put(5,12){\line(1,0){10}}}}
\put(-25,17){$q:q$}
\qbezier(-3,21)(2,24)(9,19)
\put(7,35){$p$}
\put(0,0){$q:q$}
        \end{picture}
    \end{minipage} \!\! \otimes [xb]\right)$ $=$ $h_{2}(S_{+-}(p, q) \otimes [xb])$.  
\end{lem}
\begin{proof}
$h_{2} : $ $S_{+-}(p, q) \otimes [xb]$ $\mapsto$ $S_{++}(p, q) \otimes [x]$, $S_{--}(p, q) \otimes [xab]$ $\mapsto$ $- S_{+-,-}(p, q) \otimes [xb]$ and otherwise $\mapsto 0$ \cite{ito}.  On the other hand, we have (\ref{1to1eq}).  
\end{proof}

Third, let us show the equation depending on all the relation in Figure \ref{frobenius}.  
We use the homomorphism $f_{+-,-}$ defined by $f_{+-,-}(S)$ $=$ $S$ if $S$ $=$ $S_{+-,-}(p, q) \otimes [xb]$ of any $p$ and $q$, $f_{+-,-}(S)$ $=$ $0$ otherwise.  

\begin{equation}\label{2nd_c2}
\begin{split}
\left(h_{2} \circ \delta_{s, t} + \delta_{s, t} \circ h_{2} \right)(S_{+-}(p, q) \otimes [xb]) &= S_{+-,-}(\operatorname{m}(p:+), q) \otimes [xb]\\ &\quad + S_{+-}(p,q) \otimes [xb] \quad (\because (\ref{1to1eq})) \\ &\quad+ f_{+-,-}(\delta_{s, t}(S_{++}(p, q) \otimes [x]))\\ &\quad+ S_{-+}(p:q, q:p) \otimes [xa]\\
&= S_{+-}(p,q) \otimes [xb] \\ &\quad+ S_{+-,-}((p:q):(q:p), (q:p):(p:q)) \\
&\qquad \otimes [xb]\\ &\quad+ S_{-+}(p:q, q:p) \otimes [xa]\\
&=(\operatorname{id} - \rho_{2})(S_{+-}(p, q) \otimes [xb]).  
\end{split}
\end{equation}

The third equality of (\ref{2nd_c2}) follows from Lemma \ref{+--lem}.  
\begin{lem}\label{+--lem}
\begin{equation}\label{+--eq}
\begin{split}
S_{+-,-}(\operatorname{m}(p:+), q) \otimes [xb] + f_{+-,-}&(\delta_{s, t}(S_{++}(p, q) \otimes [x]))\\
&= S_{+-,-}((p:q):(q:p), (q:p):(p:q)) \otimes [xb]. 
\end{split} 
\end{equation}  
\end{lem}
\begin{proof}
Let us consider separately the cases where the component with $p$ is that of $q$ or not.  

Consider the cases where the component with $p$ is that of $q$ (two cases).  
\begin{equation}
\begin{split}
\text{LHS} &= S_{+-,-}(\operatorname{m}(-:+), -) \otimes [xb] + f_{+-,-}(\delta_{s, t}(S_{++}(-, -) \otimes [x]))\\
    &= 2 S_{+-, -}(+, +) \otimes [xb] - s S_{+-, -}(-, -) \otimes [xb]\\
    &= \text{RHS}, \\
\text{LHS} &= S_{+-,-}(\operatorname{m}(+:+), +) \otimes [xb] + f_{+-,-}(\delta_{s, t}(S_{++}(+, +) \otimes [x]))\\
    &= s S_{+-, -}(+, +) \otimes [xb] + 2 t S_{+-, -}(-, -) \otimes [xb]\\
    &= \text{RHS}.  
\end{split}
\end{equation}

Consider the cases where the component with $p$ is not that of $q$ (four cases).  

\begin{equation}
\begin{split}
\text{LHS} &= S_{+-,-}(\operatorname{m}(-:+), -) \otimes [xb] + f_{+-,-}(\delta_{s, t}(S_{++}(-, -) \otimes [x]))\\
    &= S_{+-, -}(+, -) \otimes [xb] + S_{+-, -}(-, +) \otimes [xb] - s S_{+-, -}(-, -) \otimes [xb]\\
    &= \text{RHS}, \\
\text{LHS} &= S_{+-,-}(\operatorname{m}(+:+), +) \otimes [xb] + f_{+-,-}(\delta_{s, t}(S_{++}(+, +) \otimes [x]))\\
    &= s S_{+-, -}(+, +) \otimes [xb] + t S_{+-, -}(-, +) \otimes [xb] + t S_{+-, -}(+, -) \otimes [xb]\\
    &= \text{RHS}, \\
\text{LHS} &= S_{+-,-}(\operatorname{m}(+:+), -) \otimes [xb] + f_{+-,-}(\delta_{s, t}(S_{++}(+, -) \otimes [x]))\\
    &= t S_{+-, -}(-, -) \otimes [xb] + S_{+-, -}(+, +) \otimes [xb]\\
    &= \text{RHS}, \\
\text{LHS} &= S_{+-,-}(\operatorname{m}(-:+), +) \otimes [xb] + f_{+-,-}(\delta_{s, t}(S_{++}(-, +) \otimes [x]))\\
    &= S_{+-, -}(+, +) \otimes [xb] + t S_{+-, -}(-, -) \otimes [xb].  \\
    &= \text{RHS}.  
\end{split}
\end{equation}
\end{proof}

\subsection{On the third Reidemeister invariance.  }\label{3rd}
In this section,  consider the third Reidemeister move of Khovanov homology for $\delta_{s,t}$.  
In \cite[Subsection 2.2]{ito}, a chain homotopy map, a retraction and an isomorphism for the third Reidemeister invariance of Khovanov homology for $\delta_{0, 0}$ were given.  

We extend the chain homotopy map, the retraction and the isomorphism for $\delta_{0, 0}$ to maps $h_{3}$, $\rho_{3}$ and $\operatorname{isom}_{3}$ for $\delta_{s, t}$ by replacing $p:q$ and $q:p$ defined \cite[Figure 3]{jacobsson} with the ``generalised" $p:q$ and $q:p$ defined Figure \ref{frobenius}.  We will show that these extended maps become the chain homotopy map and the retraction and the isomorphism for the third Reidemeister invariance of Khovanov homology for $\delta_{s, t}$.  

Let $a$, $b$ and $c$ be crossings, $x$ sequence of crossings with negative markers and $p$, $q$, $r$ be signs.  For a crossing with no markers or no signs in the following formulas, any markers or signs are allowed.  Let $\tilde{r}$ be $r$ unless the upper left arc is connected to one of the other arcs in the picture and let $\tilde{q}$ be $q$ unless the lower left arc is connected to one of the other arcs in the picture.  

For a link diagram $D'$ $=$ $\begin{minipage}{60pt}
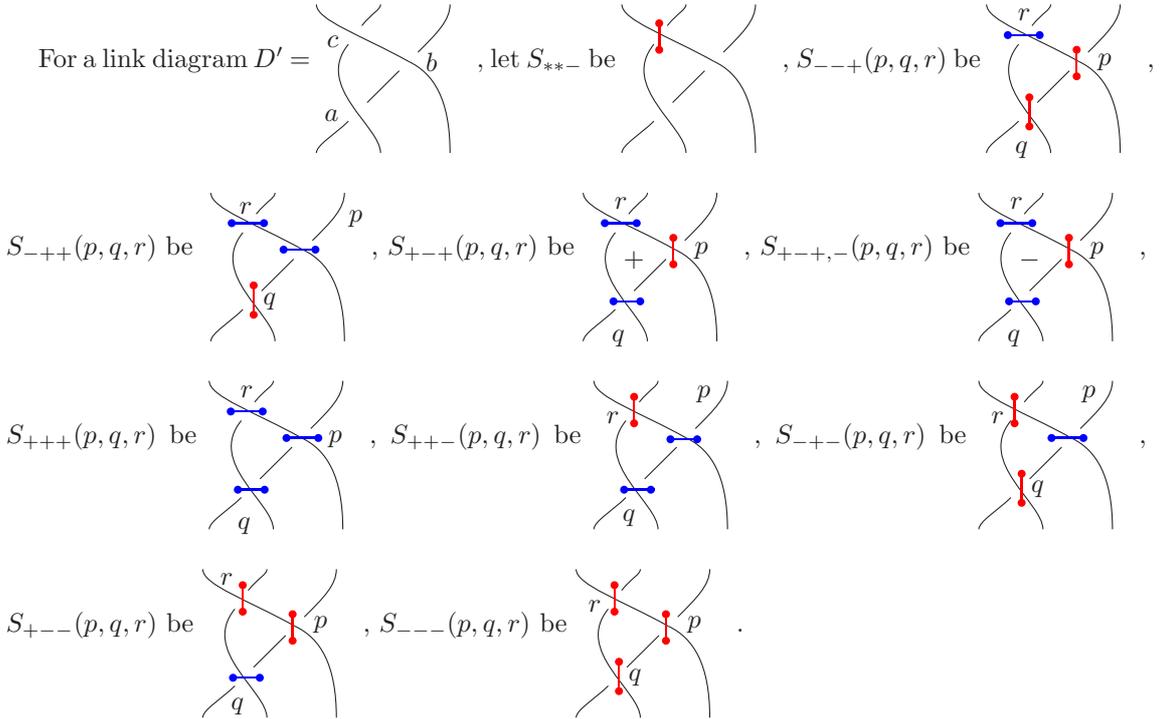
        \begin{picture}(50,70)
            \cbezier(0,0)(0,3)(7,7)(12,12)
            \qbezier(19,19)(25,25)(31,31)
            \cbezier(38,38)(45,45)(50,50)(50,56)
            \cbezier(0,56)(0,50)(17.25,44.25)(34.5,34.5)
            \cbezier(34.5,34.5)(44,30)(50,20)(50,0)
            \cbezier(17,48)(21,53)(24,54)(24,56)
            \cbezier(12,41)(-1,25)(24,10)(24,0)
            \put(4,40){$c$}
            \put(41,31){$b$}
            \put(3,12){$a$}
        \end{picture}
    \end{minipage}$, let $S_{**-}$ be $\begin{minipage}{60pt}
        \begin{picture}(50,70)
            \cbezier(0,0)(0,3)(7,7)(12,12)
            \qbezier(19,19)(25,25)(31,31)
            \cbezier(38,38)(45,45)(50,50)(50,56)
            \cbezier(0,56)(0,50)(17.25,44.25)(34.5,34.5)
            \cbezier(34.5,34.5)(44,30)(50,20)(50,0)
            \cbezier(17,48)(21,53)(24,54)(24,56)
            \cbezier(12,41)(-1,25)(24,10)(24,0)
          {\color{red}{\put(14,49){\circle*{3}}
\put(14,39){\circle*{3}}
\put(14,39){\line(0,1){10}}}}
        \end{picture}
    \end{minipage}$, $S_{--+}(p, q, r)$ be $\begin{minipage}{60pt}
        \begin{picture}(50,70)
            \cbezier(0,0)(0,3)(7,7)(12,12)
            \qbezier(19,19)(25,25)(31,31)
            \cbezier(38,38)(45,45)(50,50)(50,56)
            \cbezier(0,56)(0,50)(17.25,44.25)(34.5,34.5)
            \cbezier(34.5,34.5)(44,30)(50,20)(50,0)
            \cbezier(17,48)(21,53)(24,54)(24,56)
            \cbezier(12,41)(-1,25)(24,10)(24,0)
            {\color{red}{\put(16,21){\circle*{3}}
\put(16,10){\circle*{3}}
\put(16,11){\line(0,1){10}}}}
{\color{blue}
            {\put(4.5,44.5){\circle*{3}}
\put(16.5,44.5){\circle*{3}}
\put(5.5,44.5){\line(1,0){10}}
            }}
            {\color{red}{\put(27,39){\circle*{3}}
\put(27,29){\circle*{3}}
\put(27,29){\line(0,1){10}}}}
            \put(4,0){$q$}
            \put(35,33){$p$}
            \put(5,50){$r$}
        \end{picture}
    \end{minipage}$, $S_{-++}(p, q, r)$ be\ \ 
    $\begin{minipage}{60pt}\begin{picture}(50,70)
            \cbezier(0,0)(0,3)(7,7)(12,12)
            \qbezier(19,19)(25,25)(31,31)
            \cbezier(38,38)(45,45)(50,50)(50,56)
            \cbezier(0,56)(0,50)(17.25,44.25)(34.5,34.5)
            \cbezier(34.5,34.5)(44,30)(50,20)(50,0)
            \cbezier(17,48)(21,53)(24,54)(24,56)
            \cbezier(12,41)(-1,25)(24,10)(24,0)
            {\color{red}{\put(16,21){\circle*{3}}
\put(16,10){\circle*{3}}
\put(16,11){\line(0,1){10}}}}
{\color{blue}
            {\put(4.5,44.5){\circle*{3}}
\put(16.5,44.5){\circle*{3}}
\put(5.5,44.5){\line(1,0){10}}
            }}
            {\color{blue}
            {\put(32.5,34.5){\circle*{3}}
\put(20.5,34.5){\circle*{3}}
\put(21.5,34.5){\line(1,0){10}}
            }}
            \put(13,14){$q$}
            \put(45,45){$p$}
            \put(4,48){$r$}
        \end{picture}
    \end{minipage}$, 
    $S_{+-+}(p, q, r)$ be $\begin{minipage}{60pt}
        \begin{picture}(50,70)
            \cbezier(0,0)(0,3)(7,7)(12,12)
            \qbezier(19,19)(25,25)(31,31)
            \cbezier(38,38)(45,45)(50,50)(50,56)
            \cbezier(0,56)(0,50)(17.25,44.25)(34.5,34.5)
            \cbezier(34.5,34.5)(44,30)(50,20)(50,0)
            \cbezier(17,48)(21,53)(24,54)(24,56)
            \cbezier(12,41)(-1,25)(24,10)(24,0)
            \put(15,28){\text{$+$}}
            {\color{blue}
            {\put(8,44.5){\circle*{3}}
\put(20,44.5){\circle*{3}}
\put(9,44.5){\line(1,0){10}}
            }}
            {\color{blue}{\put(8,15){\circle*{3}}
\put(18,15){\circle*{3}}
\put(9,15){\line(1,0){10}}}}
{\color{red}{\put(27,39){\circle*{3}}
\put(27,29){\circle*{3}}
\put(27,29){\line(0,1){10}}}}
\put(4,0){$q$}
\put(35,33){$p$}
\put(5,50){$r$}
        \end{picture}
    \end{minipage}$, 
    $S_{+-+,-}(p, q, r)$ be $\begin{minipage}{60pt}
        \begin{picture}(50,70)
            \cbezier(0,0)(0,3)(7,7)(12,12)
            \qbezier(19,19)(25,25)(31,31)
            \cbezier(38,38)(45,45)(50,50)(50,56)
            \cbezier(0,56)(0,50)(17.25,44.25)(34.5,34.5)
            \cbezier(34.5,34.5)(44,30)(50,20)(50,0)
            \cbezier(17,48)(21,53)(24,54)(24,56)
            \cbezier(12,41)(-1,25)(24,10)(24,0)
             \put(15,28){\text{$-$}}
            {\color{blue}
            {\put(8,44.5){\circle*{3}}
\put(20,44.5){\circle*{3}}
\put(9,44.5){\line(1,0){10}}
            }}
            {\color{blue}{\put(8,15){\circle*{3}}
\put(18,15){\circle*{3}}
\put(9,15){\line(1,0){10}}}}
{\color{red}{\put(27,39){\circle*{3}}
\put(27,29){\circle*{3}}
\put(27,29){\line(0,1){10}}}}
\put(4,0){$q$}
\put(35,33){$p$}
\put(5,50){$r$}
        \end{picture}
    \end{minipage}$, 
    $S_{+++}(p, q, r)$ be $\begin{minipage}{60pt}
            \begin{picture}(50,70)
            \cbezier(0,0)(0,3)(7,7)(12,12)
            \qbezier(19,19)(25,25)(31,31)
            \cbezier(38,38)(45,45)(50,50)(50,56)
            \cbezier(0,56)(0,50)(17.25,44.25)(34.5,34.5)
            \cbezier(34.5,34.5)(44,30)(50,20)(50,0)
            \cbezier(17,48)(21,53)(24,54)(24,56)
            \cbezier(12,41)(-1,25)(24,10)(24,0)
{\color{blue}
            {\put(8,44.5){\circle*{3}}
\put(20,44.5){\circle*{3}}
\put(9,44.5){\line(1,0){10}}
            }}
            {\color{blue}
            {\put(37.5,34.5){\circle*{3}}
\put(25.5,34.5){\circle*{3}}
\put(26.5,34.5){\line(1,0){10}}
            }}
            {\color{blue}{\put(4,15){\circle*{3}}
\put(14,15){\circle*{3}}
\put(5,15){\line(1,0){10}}}}
\put(4,0){$q$}
\put(38,33){$p$}
\put(5,50){$r$}
        \end{picture}
    \end{minipage}$, 
    $S_{++-}(p, q, r)$ be $
\begin{minipage}{60pt}
        \begin{picture}(50,70)
            \cbezier(0,0)(0,3)(7,7)(12,12)
            \qbezier(19,19)(25,25)(31,31)
            \cbezier(38,38)(45,45)(50,50)(50,56)
            \cbezier(0,56)(0,50)(17.25,44.25)(34.5,34.5)
            \cbezier(34.5,34.5)(44,30)(50,20)(50,0)
            \cbezier(17,48)(21,53)(24,54)(24,56)
            \cbezier(12,41)(-1,25)(24,10)(24,0)
            {\color{red}
            {\put(15,50){\circle*{3}}
\put(15,40){\circle*{3}}
\put(15,40){\line(0,1){10}}
            }}
            {\color{blue}{\put(8,15){\circle*{3}}
\put(18,15){\circle*{3}}
\put(9,15){\line(1,0){10}}}}
{\color{blue}{\put(32,34){\circle*{3}}
\put(22,34){\circle*{3}}
\put(22,34){\line(1,0){10}}}}
\put(4,3){$q$}
\put(32,50){$p$}
\put(-2,40){$r$}
        \end{picture}
    \end{minipage}
$, $S_{-+-}(p, q, r)$ be $\begin{minipage}{60pt}
        \begin{picture}(50,70)
            \cbezier(0,0)(0,3)(7,7)(12,12)
            \qbezier(19,19)(25,25)(31,31)
            \cbezier(38,38)(45,45)(50,50)(50,56)
            \cbezier(0,56)(0,50)(17.25,44.25)(34.5,34.5)
            \cbezier(34.5,34.5)(44,30)(50,20)(50,0)
            \cbezier(17,48)(21,53)(24,54)(24,56)
            \cbezier(12,41)(-1,25)(24,10)(24,0)
{\color{red}{\put(16,21){\circle*{3}}
\put(16,10){\circle*{3}}
\put(16,11){\line(0,1){10}}}}
{\color{red}
            {\put(10,50){\circle*{3}}
\put(10,40){\circle*{3}}
\put(10,40){\line(0,1){10}}
            }}
            {\color{blue}
            {\put(32.5,34.5){\circle*{3}}
\put(20.5,34.5){\circle*{3}}
\put(21.5,34.5){\line(1,0){10}}
            }}
            \put(13,14){$q$}
            \put(32,50){$p$}
            \put(-2,40){$r$}
        \end{picture}
    \end{minipage}$, $S_{+--}(p, q, r)$ be $
\begin{minipage}{60pt}
        \begin{picture}(50,70)
            \cbezier(0,0)(0,3)(7,7)(12,12)
            \qbezier(19,19)(25,25)(31,31)
            \cbezier(38,38)(45,45)(50,50)(50,56)
            \cbezier(0,56)(0,50)(17.25,44.25)(34.5,34.5)
            \cbezier(34.5,34.5)(44,30)(50,20)(50,0)
            \cbezier(17,48)(21,53)(24,54)(24,56)
            \cbezier(12,41)(-1,25)(24,10)(24,0)
            {\color{red}
            {\put(15,50){\circle*{3}}
\put(15,40){\circle*{3}}
\put(15,40){\line(0,1){10}}
            }}
            {\color{blue}{\put(8,15){\circle*{3}}
\put(18,15){\circle*{3}}
\put(9,15){\line(1,0){10}}}}
{\color{red}{\put(27,39){\circle*{3}}
\put(27,29){\circle*{3}}
\put(27,29){\line(0,1){10}}}}
\put(4,3){$q$}
\put(35,34){$p$}
\put(0,50){$r$}
        \end{picture}
    \end{minipage}
$, $S_{---}(p, q, r)$ be $
\begin{minipage}{60pt}
        \begin{picture}(50,70)
            \cbezier(0,0)(0,3)(7,7)(12,12)
            \qbezier(19,19)(25,25)(31,31)
            \cbezier(38,38)(45,45)(50,50)(50,56)
            \cbezier(0,56)(0,50)(17.25,44.25)(34.5,34.5)
            \cbezier(34.5,34.5)(44,30)(50,20)(50,0)
            \cbezier(17,48)(21,53)(24,54)(24,56)
            \cbezier(12,41)(-1,25)(24,10)(24,0)
{\color{red}{\put(16,21){\circle*{3}}
\put(16,10){\circle*{3}}
\put(16,11){\line(0,1){10}}}}
{\color{red}{\put(11,50){\circle*{3}}
\put(11,40){\circle*{3}}
\put(11,40){\line(0,1){10}}}}
{\color{red}{\put(27,39){\circle*{3}}
\put(27,29){\circle*{3}}
\put(27,29){\line(0,1){10}}}}
\put(13,14){$q$}
\put(35,34){$p$}
\put(-2,40){$r$}            
\end{picture}
\end{minipage}
$.  
 
The third Reidemeister move is $D'$ $\stackrel{3}{\sim}$ $ \begin{minipage}{60pt}
        \begin{picture}(50,70)
        \cbezier(37.5,42.5)(43,47)(48,52)(50,56)
            \qbezier(31,36)(25,30)(19,24)
            \cbezier(12,17)(5,10)(0,5)(0,0)
            \cbezier(50,0)(50,5)(32.75,10.75)(15.5,20.5)
            \cbezier(15.5,20.5)(6,25)(0,50)(0,56)
            \cbezier(31.5,8.5)(30,7)(24,4)(24,0)
            \cbezier(37,13)(61,30)(24,40)(24,56)
            \put(39,9.5){$c$}
            \put(5,17){$a$}
            \put(24,35){$b$}
        \end{picture}
    \end{minipage}$ $=$ $D$, we consider the composition 

\begin{equation}
\mathcal{C}(D') = \mathcal{C'} \oplus \mathcal{C'}_{\rm{contr}} \stackrel{\rho_{3}}{\to} \mathcal{C'} \stackrel{{\rm isom_{3}}}{\to} \mathcal{C} \stackrel{\operatorname{in}}{\to} \mathcal{C}(D)
\end{equation} 
where $\mathcal{C'}$, $\mathcal{C'}_{\rm{contr}}$, $\rho_{3}$, $\mathcal{C}$ and the $\operatorname{isom}_{3}$ are defined in the following formulas in \cite[Section 2.2]{ito} by replacing $p:q$ and $q:p$ in \cite[Section 2.1]{ito} with the ``generalised" $p:q$ and $q:p$ defined Figure \ref{frobenius}.  

Let $h_{3}$ be the chain homotopy maps given by \cite[Section 2.2, Equation (12)]{ito}.  Similarly to the previous section, to verify the third Reidemeister invariance of $\mathcal{H}^{i}(D)$, it is sufficient to show that $\delta_{s, t} \circ \rho_{3}$ $=$ $\rho_{3} \circ \delta_{s, t}$ and $\delta_{s, t}$ $\circ$ $h_{3}$ $+$  $h_{3}$ $\circ$ $\delta_{s, t}$ $=$ $\operatorname{id} - \operatorname{in} \circ \rho_{3}$.  

$\bullet$ On the proof of the $\delta_{s, t} \circ \rho_{3}$ $=$ $\rho_{3} \circ \delta_{s, t}$.  

We denote a big sum of the states $\sum_{u} S \otimes [xu]$ by $S^{u}$ where $u$ ranges over crossings of $D'$ except for $a$, $b$, and $c$.  

First, we can verify that we have the following relations: $(\rho_{3} \circ \delta_{s, t})(S_{--+}(p, q, r) \otimes [xab])$ $=$ $\delta_{s, t} \circ \rho_{3}(S_{--+}(p, q, r) \otimes [xab])$ $=$ $S_{---}(p, q:r, r:q) \otimes [xabc]$ $+$ $\mathcal{S}^{u}_{+--}$; $(\rho_{3} \circ \delta_{s, t})(S_{+++}(p, q, r) \otimes [x])$ $=$ $(\delta_{s, t} \circ \rho_{3})(S_{+++}(p, q, r) \otimes [x])$ $=$ $0$; $(\rho_{3} \circ \delta_{s, t})(S_{+-+, -}(p, q, r) \otimes [xb])$ $=$ $(\delta_{s, t} \circ \rho_{3})(S_{+-+, -}(p, q, r) \otimes [xb])$ $=$ $0$; $(\rho_{3} \circ \delta_{s, t})(S_{**-} \otimes [x])$ $=$ $(\delta_{s, t} \circ \rho_{3})(S_{**-} \otimes [x])$ $=$ $\delta_{s, t}(S_{**-} \otimes [x])$.

Second, $(\rho_{3} \circ \delta_{s, t})(S_{-++}(p, q, r) \otimes [xa])$ $=$ $(\delta_{s, t} \circ \rho_{3})(S_{-++}(p, q, r) \otimes [xa])$ $=$ $S_{+--}(p:q, q:p, \tilde{r}) \otimes [xbc]$ $+$ $S_{-+-}(p:r, \tilde{q}, r:p) \otimes [xac]$ $+$ $\rho_{3}(\mathcal{S}^{u}_{-++})$ follows from (\ref{f2})--(\ref{f4}).  

In the following, the final case follows from (\ref{+--eq}).  
Let $f_{*}$ be the homomorphism such that $f_{*}(S)$ $=$ $S$ if $S$ $=$ $S_{*}(p, q, r) \otimes [x]$ of any $p$, $q$, $r$, and $x$, $f_{*}(S)$ $=$ $0$ otherwise.  
\begin{equation}
\begin{split}
(\rho_{3} \circ \delta_{s, t})(S_{+-+}(p, q, r) \otimes [xb])
&= \rho_{3}(S_{+--}(p, q, \operatorname{m}(r, +)) \otimes [xbc]) \\ 
&\quad - \rho_{3}(S_{--+}(p, \operatorname{m}(q, +), r)) + \rho_{3}(\mathcal{S}^{u}_{+-+})\\
&= S_{+--}(p, q, \operatorname{m}(r, +) \otimes [xbc]) \\ &\quad
- S_{+--}(p, \operatorname{m}(q, +), r) \otimes [xbc]) + \rho_{3}(\mathcal{S}^{u}_{+-+}).  
\end{split}
\end{equation}
On the other hand, 
\begin{equation*}
\begin{split}
(\delta_{s, t} \circ \rho_{3})(S_{+-+}(p, q, r) \otimes [xb]) &= \delta_{s, t}(- S_{-++}(q:p, p:q, r) \otimes [xa] \\ & \quad - S_{+-+, -}((p:q):(q:p), (q:p):(p:q), r) \otimes [xb]\\& \quad  - S_{++-}(p:r, q, r:p) \otimes [xc])\\
&= \delta_{s, t}\Bigl(-S_{+-+, -}((p:q):(q:p), (q:p):(p:q), r)\\& \quad \otimes [xb]\\
& \quad - \delta_{s, t} (S_{+++}(p, q, r) \otimes [x]) + 
\qquad \begin{minipage}{60pt}
\begin{picture}(50,70)
\cbezier(0,0)(0,3)(7,7)(12,12)
\qbezier(19,19)(25,25)(31,31)
\cbezier(38,38)(45,45)(50,50)(50,56)
\cbezier(0,56)(0,50)(17.25,44.25)(34.5,34.5)
\cbezier(34.5,34.5)(44,30)(50,20)(50,0)
\cbezier(17,48)(21,53)(24,54)(24,56)
\cbezier(12,41)(-1,25)(24,10)(24,0)
{\color{blue}
{\put(8,44.5){\circle*{3}}
\put(20,44.5){\circle*{3}}
\put(9,44.5){\line(1,0){10}}
}}
{\color{red}{\put(8,15){\circle*{3}}
\put(18,15){\circle*{3}}
\put(9,15){\line(1,0){10}}}}
{\color{blue}{\put(27,39){\circle*{3}}
\put(27,29){\circle*{3}}
\put(27,29){\line(0,1){10}}}}
\put(4,0){$q$}
\put(35,33){$p:p$}
\put(5,50){$r$}
\put(-27,28){$p:p$}
\qbezier(-3,31)(2,34)(9,29)
\end{picture}
\end{minipage}\\
&\quad \otimes [xb]\Bigl) \quad(\because \mathcal{S}^{u}_{+++} = 0 ~{\text{on}}~\mathcal{C}')\\
&= \delta_{s, t}\Big(- S_{+-+, -}(p, \operatorname{m}(+, q), r) \otimes [xb]\\ & \quad - f_{+-+, -}(\delta_{s, t} (S_{+++}(p, q, r) \otimes [x])) \\
\end{split}
\end{equation*}
\begin{equation}
\begin{split}
& \quad + \qquad \begin{minipage}{60pt}
\begin{picture}(50,70)
\cbezier(0,0)(0,3)(7,7)(12,12)
\qbezier(19,19)(25,25)(31,31)
\cbezier(38,38)(45,45)(50,50)(50,56)
\cbezier(0,56)(0,50)(17.25,44.25)(34.5,34.5)
\cbezier(34.5,34.5)(44,30)(50,20)(50,0)
\cbezier(17,48)(21,53)(24,54)(24,56)
\cbezier(12,41)(-1,25)(24,10)(24,0)
{\color{blue}
{\put(8,44.5){\circle*{3}}
\put(20,44.5){\circle*{3}}
\put(9,44.5){\line(1,0){10}}
}}
{\color{blue}{\put(8,15){\circle*{3}}
\put(18,15){\circle*{3}}
\put(9,15){\line(1,0){10}}}}
{\color{red}{\put(27,39){\circle*{3}}
\put(27,29){\circle*{3}}
\put(27,29){\line(0,1){10}}}}
\put(4,0){$q$}
\put(35,33){$p:p$}
\put(5,50){$r$}
\put(-27,28){$p:p$}
\qbezier(-3,31)(2,34)(9,29)
\end{picture}
\end{minipage}\ \ \ \otimes [xb]\Bigl) \quad (\because (\ref{+-+,-eq}))\\
&= \delta_{s, t}(- S_{+-+, -}(p, \operatorname{m}(+, q), r) \otimes [xb]\\
&\quad + S_{+-+}(p, q, r) \otimes [xb])\\
&= - S_{+--}(p, \operatorname{m}(+, q), r) \otimes [xbc]\\&\quad + S_{+--}(p, q, \operatorname{m}(r, +)) \otimes [xbc] + \rho_{3}(\mathcal{S}^{u}_{+-+})
\end{split}
\end{equation}
\[(\because \rho_{3}(S_{+-+}(p, q, r) \otimes [xb])  = S_{+-+}(p, q, r) \otimes [xb] - S_{+-+, -}(p, \operatorname{m}(+, q), r) \otimes [xb]~{\text{on}}~\mathcal{C}')\]
where \begin{equation}\label{+-+,-eq}
\begin{split}
S_{+-+, -}((p:q):(q:p), (q:p):(p:q), r) \otimes [xb] &= S_{+-+, -}(p, \operatorname{m}(+:q), r) \otimes [xb]\\
&\quad + f_{+-+, -}(\delta_{s, t}(S_{+++}(p, q, r) \otimes [x])).  
\end{split}
\end{equation}
(\ref{+-+,-eq}) follows from (\ref{+--eq}).  When we  localise the problem to two signed circles concerning with $p$, $q$ and exchange $p$, $q$, (\ref{+--eq}) implies (\ref{+-+,-eq}).  

$\bullet$ On the proof of the $\delta_{s, t}$ $\circ$ $h_{3}$ $+$  $h_{3}$ $\circ$ $\delta_{s, t}$ $=$ $\operatorname{id} - \operatorname{in} \circ \rho_{3}$.  

In the beginning, let us show the equation not depending on Figure \ref{frobenius}.  

\begin{equation}
\begin{split}
\left(h_{3} \circ \delta_{s, t} + \delta_{s, t} \circ h_{3} \right) (S_{**-} \otimes [x]) &= 0 \\
&=(\operatorname{id} - \rho_{3})(S_{**-} \otimes [x]), \\ 
\left(h_{3} \circ \delta_{s, t} + \delta_{s, t} \circ h_{3} \right) (S_{-++}(p, q, r) \otimes [xa]) &= h_{3}(S_{--+} \otimes [xab])\\
&= - S_{+-+, -}(q:p, p:q, \tilde{r}) \otimes [xb]\\ 
&= (\operatorname{id} - \rho_{3})(S_{-++}(p, q, r) \otimes [xa]).  
\end{split}
\end{equation}

Second, let us show the equation depending only on (\ref{f2})--(\ref{f4}) of Figure \ref{frobenius}.  

\begin{equation}
\begin{split}
\left(h_{3} \circ \delta_{s, t} + \delta_{s, t} \circ h_{3} \right) (S_{--+}(p, q, r) \otimes [xab]) &= S_{--+}(p, q, r) \otimes [xab] - S_{+--}(p, q, r) \\ &\qquad\otimes [xbc]\\ 
&= (\operatorname{id} - \rho_{3})(S_{--+} \otimes [xab]), \\
\left(h_{3} \circ \delta_{s, t} + \delta_{s, t} \circ h_{3} \right) (S_{+-+, -}(p, q, r) \otimes [xb]) &= h(S_{--+}(p, q, r) \otimes [xba])\\
&= S_{+-+, -}(p, q, r) \otimes [xb]\\ 
&= (\operatorname{id} - \rho_{3})(S_{+-+, -}(p, q, r) \otimes [xb]).  
\end{split}
\end{equation}

Third, let us show the equation depending only on (\ref{f5})--(\ref{f6}) of Figure \ref{frobenius}.   

\begin{equation}\label{3rd1to1}
\begin{split}
(h_{3} \circ \delta_{s, t} + \delta_{s, t} \circ h_{3})(S_{+++}(p, q, r) \otimes [x])
&= h_{3}\left(\qquad\begin{minipage}{60pt}
        \begin{picture}(50,70)
            \cbezier(0,0)(0,3)(7,7)(12,12)
            \qbezier(19,19)(25,25)(31,31)
            \cbezier(38,38)(45,45)(50,50)(50,56)
            \cbezier(0,56)(0,50)(17.25,44.25)(34.5,34.5)
            \cbezier(34.5,34.5)(44,30)(50,20)(50,0)
            \cbezier(17,48)(21,53)(24,54)(24,56)
            \cbezier(12,41)(-1,25)(24,10)(24,0)
            {\color{blue}
            {\put(8,44.5){\circle*{3}}
\put(20,44.5){\circle*{3}}
\put(9,44.5){\line(1,0){10}}
            }}
            {\color{blue}{\put(8,15){\circle*{3}}
\put(18,15){\circle*{3}}
\put(9,15){\line(1,0){10}}}}
{\color{red}{\put(27,39){\circle*{3}}
\put(27,29){\circle*{3}}
\put(27,29){\line(0,1){10}}}}
\put(4,0){$q$}
\put(35,33){$p:p$}
\put(5,50){$r$}
\put(-27,28){$p:p$}
\qbezier(-3,31)(2,34)(9,29)
        \end{picture}
    \end{minipage}\ \ \ \otimes [xb]\right)\\
&= S_{+++}(p, q, r) \otimes [x]\\
&= (\operatorname{id} - \rho_{3})(S_{+++}(p, q, r) \otimes [x]).  
\end{split}
\end{equation}
The second equality of (\ref{3rd1to1}) follows from Lemma \ref{1to1} by exchanging $p$, $q$.  

Fourth, let us show the equation depending on all the relation in Figure \ref{frobenius}.  

\begin{equation}\label{+-+eq}
\begin{split}
(h_{3} \circ \delta_{s, t} + \delta_{s, t} \circ h_{3})(S_{+-+}(p, q, r) \otimes [xb]) &= S_{+-+, -}(p, +:q, r) \otimes [xb]\\
& \quad + S_{-++}(q:p, p:q, r) \otimes [xa] \\
& \quad + S_{+-+}(p, q, r) \otimes [xb] \quad (\because (\ref{1to1eq})) \\
& \quad + S_{++-}(p:r, q, r:p) \otimes [xc] \\
& \quad + f_{+-+, -}(\delta_{s, t}(S_{+++} \otimes [x])) \\
&= S_{-++}(q:p, p:q, r) \otimes [xa]\\
& \quad S_{+-+}(p, q, r) \otimes [xb]\\
& \quad S_{++-}(p:r, q, r:p) \otimes [xc]\\
& \quad S_{+-+, -}((p:q):(q:p), (q:p):(p:q), r) \otimes [xb]\\
&= (\operatorname{id} - \rho_{3})(S_{+-+}(p, q, r) \otimes [xb]).  
\end{split}
\end{equation}

The third equality of (\ref{+-+eq}) follows from (\ref{+-+,-eq}).

Section \ref{reide} implies the following.  
\begin{thm}\label{t1}
For the differential $\delta_{s, t}$, $\mathcal{H}^{i}(D')$ $\cong$ $\mathcal{H}^{i}(D)$ for $D'$ $\stackrel{\mathcal{J}}{\sim}$ $D$ is given by the chain homotopy map $h_{\mathcal{J}}$ and the retraction $\rho_{\mathcal{J}}$.  
\end{thm}

\begin{thm}\label{t2}
Let $\delta$ be a given differential : $\mathcal{C}^{i}$ $\to$ $\mathcal{C}^{i+1}$ defined in \cite[Section 2]{jacobsson} except for replacing the Frobenius calculus \cite[Figure 3]{jacobsson} in the definition of the differential by any calculus satisfying (\ref{dif}).  
If (\ref{dif}) satisfies (\ref{f2})--(\ref{f4}), there is a cohomology $\mathcal{H}^{i}$ derived from $\delta$ preserving Reidemeister I.  If (\ref{dif}) satisfies (\ref{f2})--(\ref{f4}), (\ref{1to1eq}) and (\ref{+--eq}), there is a cohomology $\mathcal{H}^{i}$ preserving Reidemeister I\!I and I\!I\!I.  
\end{thm}

\vspace{1mm}
\scriptsize{\textsc{Department of Pure and Applied Mathematics Waseda University.  Tokyo 169-8555, Japan.  }}

\scriptsize{{\it E-mail address:} \texttt{noboru@moegi.waseda.jp}}
\end{document}